%% file: manuscript.tex
\documentclass[11pt]{article}
\usepackage{soul}
\usepackage{array}
\usepackage{tabularx}
\usepackage[margin=1in]{geometry}
\usepackage{amsmath,amssymb,amsthm,mathtools}
\usepackage{tikz}
\usetikzlibrary{arrows.meta, calc,positioning}
\tikzset{
  panelbox/.style={draw,line width=0.75pt,rounded corners=2pt},
  paneltitle/.style={font=\bfseries\small,align=center},
  axislab/.style={font=\scriptsize,align=center},
  tinylabel/.style={font=\tiny,align=center},
  bodylabel/.style={font=\scriptsize,align=center,text width=4.35cm},
  coeff/.style={circle,draw=black,fill=black,line width=0.3pt,inner sep=0pt},
  levelbar/.style={draw=black,fill=black,line width=0.3pt},
  guide/.style={densely dotted,line width=0.35pt},
  profile/.style={line width=1.05pt},
  commonbox/.style={draw,rounded corners=2pt,line width=0.7pt,
    font=\scriptsize,align=center,inner xsep=6pt,inner ysep=4pt},
  flowarrow/.style={-{Stealth[length=1.7mm,width=1.15mm]},line width=0.65pt}
}
\usepackage{subcaption}
\usepackage{algorithm}
\usepackage{algpseudocode}
\usepackage{booktabs}
\usepackage{enumitem}
\usepackage{xcolor}
\usepackage{microtype}
\usepackage[round,authoryear,sort,compress]{natbib}
\usepackage{authblk}
\usepackage[colorlinks=true,linkcolor=black,citecolor=blue,urlcolor=blue]{hyperref}
\usepackage[nameinlink,capitalise]{cleveref}
\graphicspath{{figures/}}

\newtheorem{theorem}{Theorem}[section]
\newtheorem{lemma}{Lemma}[section]
\newtheorem{proposition}{Proposition}[section]
\newtheorem{corollary}{Corollary}[section]
\newtheorem{conjecture}{Conjecture}[section]

\theoremstyle{definition}
\newtheorem{definition}{Definition}[section]
\theoremstyle{remark}

\newcommand{\R}{\mathbb{R}}
\newcommand{\E}{\mathbb{E}}
\newcommand{\Z}{\mathbb{Z}}
\newcommand{\Oo}{\mathcal{O}}
\newcommand{\Ss}{\mathcal{S}}
\newcommand{\Pp}{\mathbb{P}}
\newcommand{\e}{\varepsilon}

\newcommand{\ret}{\mathrm{ret}}
\DeclareMathOperator{\disc}{disc}
\DeclareMathOperator{\dist}{\mathrm{dist}}

\title{Balancing fractional Brownian motion} 
\date{\today}

\author[1]{Hengrui Luo\footnote{Alphabetical order.}}
\author[2]{Yiming Xu}
\affil[1]{Department of Statistics, Rice University}
\affil[2]{Department of Mathematics, University of Kentucky}

\begin{document}
\maketitle

\begin{abstract}
We study the discrepancy of balancing $n$ independent sample paths of fractional Brownian motion with Hurst exponent $H\in(0,1)$ on $[0,1]$, an infinite-dimensional analogue of balancing Gaussian vectors. We establish a phase transition at $H=1/2$: with high probability, the discrepancy is $\Omega(n^{1/2-H})$ and $\mathcal O(n^{1/2-H}(\log n)^{c(H)})$, where $c(H)=H+1/2$ if $H\geq 1/2$ and $c(H)=1/2$ otherwise. At the critical exponent $H=1/2$, we show that the discrepancy is $\Theta(1)$ with constant probability as $n\to\infty$. In this regime, we further characterize the geometry of the solution space by computing the expected number of local minima, establishing an overlap gap property near the existence threshold, and proving its absence at every diverging optimality threshold. We also give randomized polynomial-time algorithms that compute signings with discrepancy $\mathcal O(n^{1/2-H}\sqrt{\log n})$ for $H<1/2$, $\mathcal O((\log n)^{3/2})$ for $H=1/2$, and $\mathcal O(\sqrt{\log n})$ for $H>1/2$, with high probability. Our analysis combines a truncated balancing argument based on a wavelet representation of fractional Brownian motion with probabilistic methods.
\end{abstract}

\input{introduction}

\input{balance-fbm}

\input{critical-regime}

\input{solution-space}

\section{Conclusion}\label{sec:con}

In this work, we studied the function-balancing problem in the average-case setting where the functions are independent sample paths of fBm on $[0,1]$ with Hurst exponent $H\in(0,1)$. As the number of functions tends to infinity, we established high-probability upper bounds on the discrepancy together with matching lower bounds up to logarithmic factors, and identified a phase transition at the critical value $H=1/2$. At the critical point, we further sharpened the upper bound to match the lower bound with constant probability. We also characterized the geometry of the solution space in this regime by analyzing the expected number of local minima. 

Although this work established asymptotic results for balancing fBm, many interesting questions remain open. Besides the natural extension of our result to general indexing spaces other than $[0,1]$, it would be interesting to determine whether our approach in \Cref{sec:proof} can be extended to general Gaussian processes or random fields. We expect that a similar truncated balancing argument could be applied under an appropriate choice of basis expansion. For example, for smooth random fields, one may instead consider the Karhunen--Lo\`eve expansion or related kernel embeddings. These heuristics may also be connected to the discrete problem of tensor balancing. Although tensor balancing is formally equivalent to vector balancing after flattening, it becomes rather different when the analysis and algorithms exploit the low-complexity geometry of rank-one tensor directions.

On the other hand, a key obstacle to obtaining rigorous algorithmic hardness results in the critical case $H=1/2$ is that the sharp discrepancy upper bound is established only with constant probability via the second-moment method. It is therefore natural to ask whether this can be strengthened to a high-probability bound. In finite-dimensional settings such as the random symmetric perceptron model, \cite{abbe2022proof} proved that the partition function normalized by its expectation converges in distribution to a log-normal random variable. Inspired by this result, we conjecture that the same phenomenon holds for $Z_n(a)$ defined in \eqref{Zn}. Matching the unknown parameters with the limiting second moment $\E[Z^2_n(a)]/\E[Z_n(a)]^2\to 1/\sqrt{1-\beta/a^2}$ (cf.~\eqref{hr}) leads to the following conjecture. 
\begin{conjecture}\label{conj}
For $a>a_0$, as $n\to\infty$, 
\begin{align*}
\frac{Z_n(a)}{\E[Z_n(a)]}\xrightarrow{d}\mathrm{Lognormal}\left(-\frac{\lambda(a)}{2}, \lambda(a)\right), \qquad \lambda(a)\coloneqq -\frac{1}{2}\log\left(1-\frac{\beta}{a^2}\right). 
\end{align*}
\end{conjecture}

\Cref{conj} is also supported by a formal higher-moment calculation. Fix $k\geq 2$ and suppose, as suggested by the second-moment analysis in \Cref{sec:1/2}, that the dominant contribution to $\E[Z^k_n(a)]$ comes from $k$-tuples of signings whose pairwise overlaps are bounded by $\mathcal O(n^{3/5})$. A perturbation of the corresponding $k$-dimensional small-ball exponent, analogous to \eqref{bdd-3}, predicts
\begin{align*}
\frac{\E[Z^k_n(a)]}{\E[Z_n(a)]^k}\xrightarrow{\text{$n\to\infty$}}\left(\frac{1}{\sqrt{1-\beta/a^2}}\right)^{{k\choose 2}}=\exp\left\{{k\choose 2}\lambda(a)\right\}.   
\end{align*}
Indeed, the second derivative in each of the ${k\choose 2}$ correlation variables gives the same contribution as in the second-moment computation, while the mixed second derivatives vanish by the tensor-product parity structure. The right-hand side is exactly the $k$th moment of the lognormal random distribution in \Cref{conj}. However, we shall emphasize that, since the lognormal distribution is moment-indeterminate, even rigorous convergence of all fixed moments would not by itself imply convergence in distribution. Nevertheless, this formal calculation provides a consistency check for \Cref{conj}; see \Cref{fig:1} for additional numerical evidence.
\begin{figure}[htbp]
  \centering 
  \begin{subfigure}{0.40\textwidth}{\includegraphics[width=\linewidth, trim={0cm 0cm 3cm 0cm},clip]{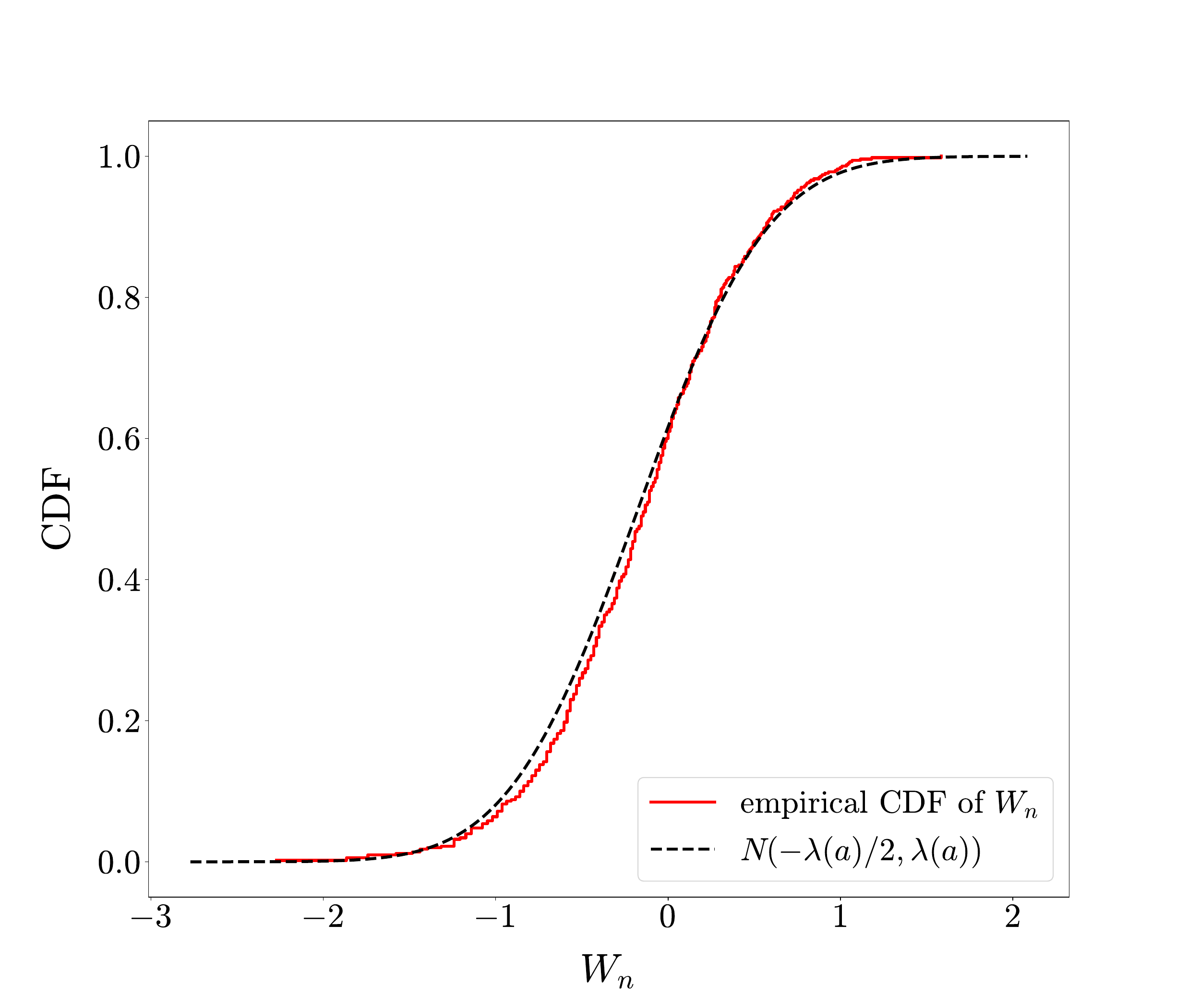}}
\end{subfigure}\hspace{2cm}
  \begin{subfigure}{0.40\textwidth}{\includegraphics[width=\linewidth, trim={0cm 0cm 3cm 0cm},clip]{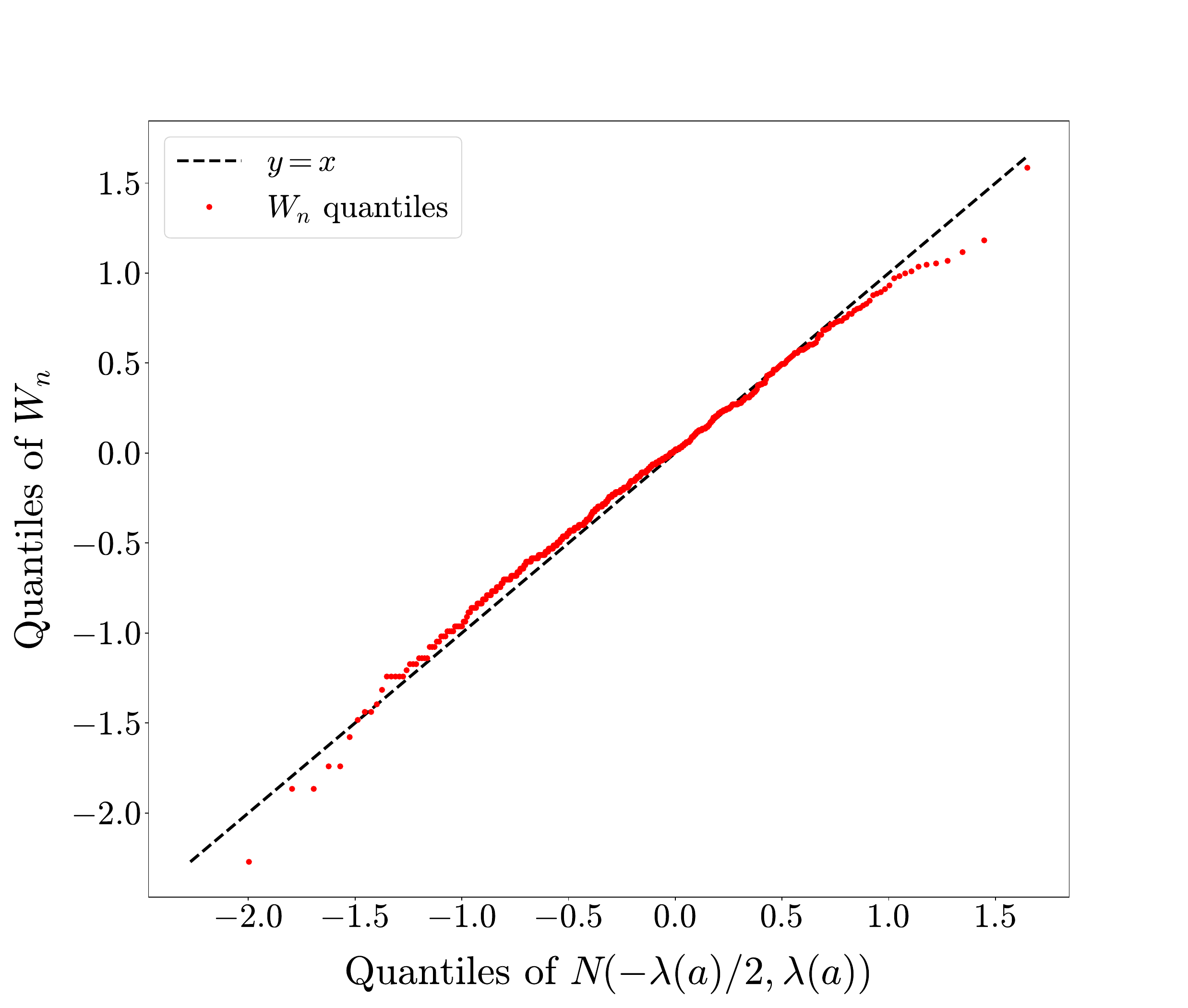}}
\end{subfigure}
\caption{Empirical CDF of $W_n\coloneqq\log Z_n(a)-\log\E[Z_n(a)]$ and the CDF of $N(-\lambda(a)/2, \lambda(a))$ (left), together with the QQ-plot (right), for $n=25$ and $a=1.4$, based on $500$ Monte Carlo simulations with $[0, 1]$ discretized into $256$ time steps. The Kolmogorov--Smirnov test in this case yields a test statistic $0.0471$ and a \texttt{p}-value $0.2097$.} \label{fig:1}
\end{figure}

Establishing \Cref{conj} would provide a natural route toward upgrading the constant-probability discrepancy bound to a high-probability result. To combine such a limit theorem with our existing 2-OGP analysis, one would need quantitative convergence estimates strong enough to imply that the probability that the discrepancy exceeds $a>a_0$ decays exponentially in $n$. One possible approach is to study the concentration of $\log Z_n(a)$ via an inductive argument similar to that in \cite{perkins2021frozen}. We leave these questions for future investigation.

\section{Acknowledgments and AI use}

H. Luo was supported in part by the U.S. Department of Energy under Contract DE-AC02-05CH11231, and NSF Grant DMS-2412403, DMS-2606336. Y. Xu was supported in part by the start-up funding from the University of Kentucky and NSF Grant DMS-2607989. 

The authors used Claude (Sonnet 5 and Fable 5) interactively to assist with developing the simulation code for \Cref{fig:1} and checking the manuscript for mathematical and typographical errors. Claude also helped suggest the observation in \eqref{b1b2}, which led us to derive a sharp lower bound on $\lambda_1(\rho)$ at both endpoints $\rho=0$ and $\rho=1$. This observation establishes the sharpness of the lower bound $a_0$ in \Cref{thm:main1}, which is crucial to our proof of \Cref{thm:ogp} and ensures that the 2-OGP result holds in a nonvacuous region. Our original approach, based on weaker estimates using the Faber--Krahn inequality, yielded a looser lower bound that was not sharp enough for \Cref{thm:ogp}. All mathematical content, including the problem formulation, theorems and proofs, interpretations, and conclusions, was developed and verified by the authors, who assume full responsibility for the contents of this work.

\bibliographystyle{unsrtnat}
\bibliography{ref}

\end{document}

%% file: introduction.tex
\section{Introduction}
\label{sec:introduction}
Given vectors $x_1, \ldots, x_n\in\R^m$, their \emph{combinatorial discrepancy} is defined by
\begin{align}
\disc(x_1, \ldots, x_n)\coloneqq\min_{\sigma\in\{\pm 1\}^n}\left\|\sum_{i=1}^n\sigma_i x_i\right\|_\infty. \label{disc-v}
\end{align}
This quantity measures how well a set of vectors can be balanced under a binary partition and is closely related to other notions such as geometric discrepancy; see, e.g.,~\cite{matousek1999geometric}. Discrepancy theory is a classical area of mathematics and theoretical computer science with numerous applications, including randomized algorithms and optimization \citep{srinivasan1999approximation, chazelle2000discrepancy, nikolov2014new}, experimental design \citep{harshaw2024balancing, raodistributional, chennonlinear}, numerical integration \citep{niederreiter1992random, bansal2025quasi}, and machine learning \citep{karnin2019discrepancy, dwivedi2024kernel, carrell2025low}.

The study of combinatorial discrepancy is typically formulated in two ways, depending on how the vectors under balancing are treated. The first assumes that the $x_i$ are fixed and seeks to understand how small the discrepancy can be in the worst case. These results are traditionally obtained nonconstructively using techniques based on partial coloring \citep{beck1981integer, spencer1985six, gluskin1989extremal} or convex geometry \citep{banaszczyk1998balancing, banaszczyk2012series}. Recently, significant progress has been made on algorithmic constructions to realize these ideas; see, for example, \cite{bansal2010constructive, lovett2015constructive, rothvoss2017constructive, bansal2018gram, eldan2018efficient, bansal2025improved} for a subset of these works, and \cite{bansal2022discrepancy} for a comprehensive review.
 
Alternatively, one may assume that the $x_i$ are drawn from a probability distribution. Intuitively, randomness can lead to a substantially smaller discrepancy on average than in the worst case due to cancellations. \citet{karmarkar1986probabilistic} gave the first average-case analysis for balancing i.i.d. random variables and showed that $\disc(x_1, \ldots, x_n)=\Oo(\sqrt{n}2^{-n})$ with high probability under suitable distributional conditions. See also \cite{borgs2009proof1, borgs2009proof2} for the detailed asymptotics characterizing the associated energy landscape. \citet{costello2009balancing} extended the result of \citet{karmarkar1986probabilistic} to the multidimensional setting for standard Gaussian vectors, obtaining $\disc(x_1, \ldots, x_n)=\Oo(\sqrt{n}2^{-n/m})$ for fixed dimension $m$. \citet{turner2020balancing} further generalized this result to the regime where $\omega(1)=m=o(n)$. The Gaussian vector balancing model is closely related to the Gaussian symmetric binary perceptron model \citep{aubin2019storage}. In this setting, a regime of particular interest is $m=\Theta(n)$, in which $\disc(x_1,\ldots,x_n)=\Theta(\sqrt{n})$ with high probability. A central problem in this regime is to determine the exact asymptotics of the discrepancy as a function of the aspect ratio $m/n$ \citep{abbe2022proof,perkins2021frozen, gamarnik2022algorithms, pmlr-v195-gamarnik23a}. Beyond Gaussian vectors, the discrepancy problem has also been studied for other discrete distributions \citep{borgs2001phase, potukuchi2018discrepancy, hoberg2019fourier, franks2020discrepancy, altschuler2022discrepancy, doi:10.1137/21M1451427}. 

Existing work on combinatorial discrepancy has predominantly focused on balancing vectors. Many modern datasets, however, are more naturally represented as functions or other infinite-dimensional objects. In applications where vectors are obtained by discretizing such continuous objects, the ambient dimension may greatly exceed the number of samples ($m \gg n$), making it more natural to study the balancing problem directly in the underlying continuous setting.
In this work, we consider a functional analogue of \eqref{disc-v} by replacing the vectors $x_i$ in \eqref{disc-v} with continuous functions $f_i$ defined on a shared compact domain $\Omega\subset\mathbb{R}^d$ with positive Lebesgue measure and define
\begin{align}
\disc(f_1,\ldots,f_n)\coloneqq\min_{\sigma\in\{\pm1\}^n}\left\|\sum_{i=1}^n\sigma_i f_i\right\|_{L^\infty(\Omega)}. \label{disc-f}
\end{align}
By the standard isometric embedding of finite-dimensional $\ell_\infty^m$ into $L^\infty(\Omega)$, every finite-dimensional instance of \eqref{disc-v} can be realized as an instance of \eqref{disc-f}. Since such an embedding holds for all $m$ simultaneously, the continuous formulation contains all finite-dimensional instances at once, and its worst-case discrepancy is trivial. In particular, choosing $2^n$ points $\{t_\sigma\}_{\sigma\in\{\pm1\}^n}\subset\Omega$ and continuous functions with $f_i(t_\sigma)=\sigma_i$ and $\|f_i\|_\infty\leq 1$ forces $\sum_i\sigma_if_i(t_\sigma)=n$ for every signing, so that $\disc(f_1,\ldots,f_n)=n$ attains the trivial maximum. Consequently, obtaining a nontrivial functional discrepancy problem requires imposing additional structure on the $f_i$.

To this end, we adopt an average-case perspective and model the $f_i$ as independent sample paths of a random process satisfying suitable regularity assumptions. Since an optimal choice of $\sigma$ ensures that the signed sum remains uniformly small over the entire domain, this leads to the following question: how does the discrepancy \eqref{disc-f} depend on the probabilistic and regularity properties of the stochastic process generating the $f_i$? Addressing this question requires quantifying how path regularity interacts with the discretization needed to control the supremum norm.

Motivated by existing work in Gaussian vector balancing, a natural model is to let $f_i$ be sample paths of Brownian motion on $\R_+$, whose law can be viewed as a canonical Gaussian measure on an infinite-dimensional Gaussian space. More generally, we consider the case where the $f_i$ are independent sample paths of fractional Brownian motion (fBm) on $\R_+$ with Hurst exponent $H\in(0,1)$ \citep{MandelbrotVanNess1968}, i.e., $f_i(t)$ is a Gaussian process with mean zero and covariance function
\begin{align}
&\E[f_i(s)f_i(t)] = \frac{1}{2}\left(s^{2H} + t^{2H} - |t-s|^{2H}\right),& t, s\geq 0. \label{fBm-cov}
\end{align}
Setting $H=1/2$ recovers the Brownian motion case, while allowing $H$ to vary yields a broader family of models with different regularities. This enables us to investigate how path regularity influences the asymptotic behavior of discrepancy.

In the rest of the article, we set $\Omega=[0, 1]$ though our analysis can be generalized to any finite interval. With slight abuse of notation, we abbreviate $\|\cdot\|_{L^\infty(\Omega)}$ as $\|\cdot\|_\infty$. We let $f_1, \ldots, f_n$ be i.i.d. sample paths of fBm with Hurst exponent $H\in(0,1)$. We will study $\disc(f_1,\ldots,f_n)$ from both the theoretical and algorithmic perspectives. Our main results are given next.  

\paragraph{Main results.} Let $a_0\coloneqq\pi/\sqrt{8\log 2}$ denote a threshold parameter that appears in our results. 
Let 
\begin{align}
h(p) = -p\log p - (1-p)\log (1-p), \qquad 0<p<1\label{bi-entropy}
\end{align}
denote the binary entropy function with natural base. 
For $\sigma\in\{\pm 1\}^n$, denote its associated loss by 
\begin{align}
L(\sigma)\coloneqq\left\|\sum_{i=1}^n\sigma_i f_i\right\|_\infty.
\end{align}
We first establish the asymptotic behavior of $\disc(f_1,\ldots,f_n)$ for all $H\in(0,1)$ as well as the loss guarantees achievable by certain efficient algorithms.
\begin{theorem}\label{thm:main}
Let 
\begin{align*}
c(H) = \begin{cases}
1/2 & H<1/2\\
1/2+H & H\geq 1/2
\end{cases}.
\end{align*}
\begin{enumerate}
\item [(1)] For all large $n$, with probability $1-o(1)$, 
\begin{align*}
n^{1/2-H}\lesssim\disc(f_1,\ldots,f_n)\lesssim n^{1/2-H}(\log n)^{c(H)},
\end{align*}
where the implicit constants depend only on $H$. When $H = 1/2$, for every fixed $\e\in(0,a_0)$, the implicit constant in the lower bound can be taken as $a_0-\e$, i.e., $\disc(f_1,\ldots,f_n)>a_0-\e$. 
\item [(2)] For every $H\in(0,1)$, there exists a randomized polynomial-time algorithm that returns a signing $\sigma\in\{\pm1\}^n$ satisfying, 
\begin{align*}
L(\sigma)\lesssim
\begin{cases}
n^{1/2-H}\sqrt{\log n} &0<H<1/2,\\
(\log n)^{3/2} &H=1/2,\\
\sqrt{\log n} &1/2<H<1,
\end{cases}
\end{align*}
with probability $1-o(1)$ over the sample paths (and the internal randomness of the algorithm), 
where the implicit constants depend only on $H$.
\end{enumerate} 
\end{theorem}
Both the theoretical and constructive upper bounds are obtained through a general truncated balancing argument based on the Meyer--Sellan--Taqqu (MST) wavelet representation of fBm \citep{MeyerSellanTaqqu1999}; see also \cite{hall1996choice,luo2026wavelet}. Specifically, we truncate the MST representation of the sample paths at a suitably chosen level and apply vector balancing discrepancy machinery to the resulting coefficient vectors or their weighted versions. For the theoretical upper bound, the case $H\geq 1/2$ is established using the existential result of \citet{turner2020balancing}, while the case $H<1/2$ relies on the constructive Gram--Schmidt walk algorithm of \citet{bansal2018gram}. Notably, the same theoretical upper bound in the regime $H<1/2$ cannot be obtained by applying the result of \citet{turner2020balancing} despite its sharpness in the average-case setting. Heuristically, this suggests that the roughness in this regime introduces sufficient effective dimensions that move beyond the regime captured by the average-case analysis. The polynomial-time guarantee assumes access to the truncated MST coefficients. In our analysis, the number of retained coefficients is quasi-linear in $n$. The lower bound is obtained via a first-moment argument. The detailed proof is deferred to \Cref{sec:proof}.

\Cref{thm:main} illustrates a phase transition at the critical exponent $H=1/2$. Our next result sharpens the estimate in the critical regime via a second-moment computation. 

\begin{theorem}\label{thm:main1}
Let $\beta$ be the positive constant defined in \eqref{beta} (i.e., $\beta\in [0.982, 0.983]$). For $H=1/2$ and any sequence $a_n$ with $a\coloneqq\inf_n a_n>a_0$, 
\begin{align}
\Pp\left\{\disc(f_1,\ldots,f_n)\leq a_n\right\}\geq \sqrt{1-\beta/a_n^2} - o(1),
\end{align}
where $o(1)$ depends only on $a$ and $n$. 
\end{theorem}
\Cref{thm:main1} suggests that $\disc(f_1,\ldots,f_n)$ is effectively of order $\Oo(1)$; the threshold $a_0$ matches the lower bound in \Cref{thm:main}. The proof relies on a technical lemma (\Cref{lm:corB}) on the uniform small-ball asymptotics for the maximum of correlated planar Brownian motion and is deferred to \Cref{sec:(ii)}. A limitation of our analysis is that the probability remains asymptotically constant for every fixed $a>a_0$, rather than converging to $1$ in probability. In general, closing such a gap is challenging, and we briefly discuss this in \Cref{sec:con}. 

For the critical regime $H = 1/2$, we further characterize the geometry of the solution space. Inspired by \cite{addario2019local, gamarnik2023algorithmic}, we first compute the expected number of local minima of $L(\sigma)$ under the Hamming distance within a given sublevel set.
\begin{definition}[One-flip local minima]
For $\sigma\in\{\pm1\}^n$, let $\sigma^{(i)}$ denote its Hamming neighbor obtained by flipping the $i$th coordinate.
A signing $\sigma$ is called a local minimum with respect to the Hamming distance if $L(\sigma)\leq \min_i L(\sigma^{(i)})$. 
\end{definition}

\begin{theorem}\label{thm:count}
Let $f_1, \ldots, f_n$ be independent sample paths of standard Brownian motion on $[0,1]$ started from the origin.
Let $Z_n(a)$ be the number of elements $\sigma$ with $L(\sigma)\leq a$, and let $Y_n(a)$ be the number of signings that are also local minima. For any sequence $a_n$ with $a\coloneqq \inf_n a_n>a_0$ and $a_n<n^{1/4}/\sqrt{4\log n}$, 
\begin{align}
\E[Y_n(a_n)] = (1-o(1))\E[Z_n(a_n)] = \left(\frac{4}{\pi}-o(1)\right)\exp\left\{\left(\log 2-\frac{\pi^2}{8a_n^2}\right)n\right\}. \label{expec}
\end{align}
If we further assume $a_n\to\infty$, with probability $1-o(1)$, $Y_n(a_n) = (1+o(1))\E[Y_n(a_n)]$ and $Z_n(a_n) = (1+o(1))\E[Z_n(a_n)]$, where $o(1)$ depends only on $a_n$ and $n$.  
\end{theorem}

The expectation asymptotics in \eqref{expec} follow from a first-moment computation, while the concentration results are obtained through a second-moment computation leveraging the estimates in the proof of \Cref{thm:main1}. The upper bound on $a_n$ is needed only for the results for $Y_n$ and can be relaxed to $a_n=o(\sqrt{n})$ for $Z_n$. The details are deferred to \Cref{sec:ogp}. Formally, \Cref{thm:count} shows that, for every sublevel threshold $a_n\to\infty$ satisfying $a_n<n^{1/4}/\sqrt{4\log n}$, asymptotically almost every level-$a_n$ configuration is a one-flip local minimum.  This reveals a locally rugged structure of the solution space of balancing Brownian motion. 

To further investigate potential geometric barriers to obtaining near-optimal signings, we study the geometry of near-optimal solutions through their (normalized) overlap, which measures the similarity of their sign patterns. A natural framework for this is the \emph{overlap gap property} (OGP), which originated in statistical physics as a tool for understanding algorithmic hardness \citep{achlioptas2006solution, mezard2005clustering}; see also \cite{gamarnik2014limits, gamarnik2018finding}. Roughly speaking, the OGP states that, with high probability, the overlaps between pairs of near-optimal solutions avoid a fixed interval. Geometrically, this implies that near-optimal solutions cluster into well-separated regions of the hypercube with respect to Hamming distance. The following definition, adapted from \cite[Definition 2.1]{gamarnik2023algorithmic}, specializes the OGP to balancing fBm.

\begin{definition}[$m$-OGP for balancing fBm]\label{def:m-OGP}
Let $\{f^{(i)}_1, \ldots, f^{(i)}_n\}_{i=0}^m$ denote $(m+1)$ independent copies of the set of $n$ independent sample paths of fBm on $[0, 1]$ with Hurst exponent $H$. For any $\tau_i\in [0, \pi/2]$, denote the interpolation $g^{(i)}_j(t; \tau_i) = \cos(\tau_i)f_j^{(0)}(t) + \sin(\tau_i)f_j^{(i)}(t)$ for $i=1,\ldots, m$ and $j = 1,\ldots, n$. Fixing $0<\eta_1<\eta_2<1$, let $\mathcal S(\eta_1, \eta_2, m, a_n, \mathcal I)$ denote the set of $m$-tuples $\{\sigma^{(i)}\}_{i=1}^m$ of $\{\pm 1\}^n$ satisfying the following conditions:  
\begin{itemize}
\item (Near-optimality condition) There exist $\tau_1, \ldots, \tau_m\in\mathcal I\subseteq [0, \pi/2]$ such that $\sigma^{(i)}$ balances the set $\{g^{(i)}_1(t; \tau_i), \ldots, g^{(i)}_n(t; \tau_i)\}$ up to order $a_n$, i.e., 
\begin{align*}
\left\|\sum_{j=1}^n\sigma^{(i)}_jg^{(i)}_j(\cdot; \tau_i)\right\|_\infty\leq a_n. 
\end{align*}
\item (Pairwise overlap condition) $\{\sigma^{(i)}\}_{i=1}^m$ have intermediate correlations within $[\eta_1, \eta_2]$: 
\begin{align*}
\frac{1}{n}\left|\sum_{\ell=1}^n\sigma^{(i)}_\ell\sigma^{(j)}_\ell\right|\in [\eta_1, \eta_2], \quad\quad \forall i\neq j.  
\end{align*} 
\end{itemize}
We say that the $m$-OGP with parameters $(\eta_1, \eta_2, m, a_n, \mathcal I)$ holds if $\Pp\{\mathcal S(\eta_1, \eta_2, m, a_n, \mathcal I)\neq\emptyset\} = \Oo(e^{-cn})$ for some absolute constant $c>0$. 
\end{definition}

For interpretation, consider the case $m=2$ with $I=[0,\pi/2]$. Suppose a deterministic algorithm finds solutions with optimality level $a_n$ along the two interpolating instance paths $\{g^{(i)}_1(\cdot, \tau), \ldots, g^{(i)}_n(\cdot; \tau)\}_\tau$ for $i=1, 2$. At $\tau=0$, the two instances coincide, so the algorithm produces the same output. At $\tau=\pi/2$, the two instances are independent. If the algorithm's outputs on independent instances have asymptotically negligible overlap, then the endpoint outputs are nearly orthogonal. If the algorithm's output changes gradually as the input is perturbed, the overlap between the two outputs must therefore pass through an intermediate interval $[\eta_1,\eta_2]$. A 2-OGP states that, with high probability, no pair of near-optimal solutions with such intermediate overlap exists. Thus, an algorithm that succeeds throughout the interpolation cannot simultaneously be stable under perturbations and decorrelate on independent instances. This provides a geometric obstruction to algorithms satisfying these properties, although additional work is required to turn it into a formal hardness theorem. The case $m>2$ generalizes this intuition to the multiple-instance setting. 

For balancing Gaussian random variables, \citet{gamarnik2023algorithmic} used the $m$-OGP for general $m$ to establish barriers for specified classes of stable algorithms up to order $2^{-\omega(n\log^{-1/5}n)}$. Here we provide similar geometric evidence on the obstruction to finding a near-optimal solution by establishing a $2$-OGP result. 

\begin{theorem}\label{thm:ogp}
For $H = 1/2$, there exist $0<\eta_1<\eta_2<1$ and $a>a_0$ such that the corresponding solution to the Brownian motion balancing problem satisfies the 2-OGP with parameters $(\eta_1, \eta_2, 2, a, [0, \pi/2])$. For instance, taking $\eta_1= 0.942$, $\eta_2 = 0.943$, and $a = 1.41$ suffices. 
\end{theorem}

The proof of \Cref{thm:ogp} follows from a first-moment computation and is deferred to \Cref{sec:ogp}. The corresponding optimality threshold $a$ exceeds the threshold $a_0$ identified by Theorems~\ref{thm:main} and~\ref{thm:main1}. However, \Cref{thm:ogp} cannot be translated into a rigorous algorithmic hardness result for stable algorithms via \Cref{thm:main1}. In fact, establishing such a result requires a positive probability that the discrepancy of all interpolation instances is below $a$, which does not follow from the constant-probability guarantee in \Cref{thm:main1}. We therefore interpret \Cref{thm:ogp} only as evidence of a geometric obstruction, rather than as a rigorous algorithmic hardness result under the current form of \Cref{thm:main1}.

One may wonder whether the first-moment computation can be extended to establish the $m$-OGP for $m=m(n)\to\infty$, thereby yielding a stronger optimality threshold $a_n$ beyond the constant level, as pursued in \citet{gamarnik2023algorithmic}. 
Our next result provides some negative evidence on this approach. In particular, for every diverging sequence of optimality level $a_n$, the single-instance solution set at $\tau=0$ contains pairs whose absolute overlaps lie in any prescribed nontrivial interval with high probability. Thus, the 2-OGP in \Cref{def:m-OGP} is absent at diverging levels when $\mathcal I = \{0\}$ (and hence for any $\mathcal I\ni 0$). This geometric statement shows that the fixed-level 2-OGP does not persist at diverging thresholds, but it does not by itself imply algorithmic tractability. 
\begin{theorem}\label{thm:ogp-n}
For any $0<\eta_1<\eta_2< 1$ and $a_n = \omega(1)$, $\Pp\{\mathcal S(\eta_1, \eta_2, 2, a_n, \{0\})\neq\emptyset\} = 1-o(1)$. 
\end{theorem}
The proof follows from a comparison argument based on \Cref{thm:count} and the forbidden intersection bound in \cite{frankl1987forbidden}. The details are given in \Cref{sec:ogp}.  

\paragraph{Organization.}
The rest of the article is organized as follows. \Cref{sec:proof} contains the proof of \Cref{thm:main}. \Cref{sec:(ii)} contains the proof of \Cref{thm:main1}. \Cref{sec:ogp} contains the proofs of Theorems~\ref{thm:count}, \ref{thm:ogp}, and \ref{thm:ogp-n}. Finally, \Cref{sec:con} concludes the article by discussing several open questions for future research.

%% file: balance-fbm.tex
\section{Balancing fBm}\label{sec:proof}

The goal of this section is to prove \Cref{thm:main}. We first provide an outline of the proof and then fill in the details in the subsequent subsections.

For the upper bounds, we apply a truncated balancing argument based on the MST representation of fBm. In particular, each $f_i$ admits the following representation: 
\begin{align}
f_i(t) = \sum_{(j, k)\in\Z\times\Z}\e_{i,j,k}G^H_{j,k}(t), \qquad \e_{i,j,k}\stackrel{\mathrm{iid}}{\sim} N(0, 1),
\end{align}
where $G^{H}_{j,k}(t)$ are defined through the fractional integrals of Meyer wavelets; see \eqref{eq:mst-kernels} for the definition. The series converges uniformly on $[0,1]$ a.s. 

We choose a suitable finite index set $\mathcal I\subset\Z\times\Z$ and truncate $f_i$ by retaining only the coefficients indexed by $\mathcal I$:
\begin{align}
f_{i, \ret}(t) = \sum_{(j, k)\in\mathcal I}\e_{i,j,k}G^H_{j,k}(t). \label{retain}
\end{align}
Denoting the retained coefficients $x_i = (\e_{i, j, k})^\top_{(j, k)\in\mathcal I}\in\R^{|\mathcal I|}$, the proof proceeds by applying vector balancing machinery to the vectors $x_i$, or weighted variants thereof, to obtain signings that balance $f_{i,\ret}$, and then transfers this balancing to the original functions $f_i$ by accounting for the truncation error. Two different vector balancing results are used depending on the dimension of the retained coefficient vector, namely $m\coloneqq |\mathcal I|$:
\begin{itemize}
\item When $m=o(n)$, we apply the average-case vector balancing result of \citet{turner2020balancing} to obtain a signing vector via an existential argument.
\item When $m=\Omega(n)$, we employ the Gram--Schmidt walk of \citet{bansal2018gram}, a constructive vector balancing approach for the worst-case setting. This procedure yields a randomized algorithm whose running time is polynomial in $n$.
\end{itemize}

The lower bounds are established via a first-moment computation which is independent of the MST representation. For every fixed signing $\sigma$, the normalized signed sum $\frac{1}{\sqrt{n}}\sum_{i=1}^n\sigma_i f_i$ has the same distribution as $f_1$ (i.e., an fBm). Consequently, the expected number of signings with discrepancy below a prescribed threshold can be upper-bounded by analyzing the corresponding small-ball probability. In the Brownian setting, this small-ball probability admits explicit asymptotics, which yields the explicit constant in the desired lower bound. A summary of our results is given in Table~\ref{tab:phase-transition-summary}. 

\begin{table}[ht]
\centering
\renewcommand{\arraystretch}{1.12}
\begin{tabular}{@{}llll@{}}
\toprule
Regime & Existential upper bound & Constructive upper bound & Lower bound \\
\midrule
$0<H<1/2$ & $\Oo(n^{1/2-H}(\log n)^{1/2+H})$ & $\Oo(n^{1/2-H}\sqrt{\log n})$ & $\Omega(n^{1/2-H})$ \\
$H=1/2$   & $\Oo(\log n)$                    & $\Oo((\log n)^{3/2})$        & $\Omega(1)$ \\
$1/2<H<1$ & $\Oo(n^{1/2-H}(\log n)^{H+1/2})$ & $\Oo(\sqrt{\log n})$         & $\Omega(n^{1/2-H})$ \\
\bottomrule
\end{tabular}
\caption{Summary of our results obtained via the existential and constructive approaches in three different regimes, together with the corresponding lower bounds. The implicit constants in $\Oo(\cdot )$ depend on $H$.}
\label{tab:phase-transition-summary}
\end{table}

The remainder of this section is devoted to the details of the proof. \Cref{subsec:wl} reviews the MST wavelet representation of fBm and provides the associated estimates. \Cref{subsec:vb} presents the existential and constructive vector balancing techniques used in our analysis. Sections~\ref{subsec:eb} and~\ref{subsec:cb} establish the upper bounds by balancing the retained functions using the existential and constructive machinery in \Cref{subsec:vb}, respectively. \Cref{subsec:lb} provides the proof of the lower bound. For convenience, we use $a_n\lesssim_H b_n$ to denote $a_n\leq C(H) b_n$ for all large $n$ for some constant $C(H)$ depending only on $H$. 

\subsection{MST wavelet representation}\label{subsec:wl}

For $f\in L^2(\R)$, we let $\widehat{f}(\xi) = \frac{1}{\sqrt{2\pi}}\int_\R f(t) e^{-i\xi t}\mathrm d t$ denote its Fourier transform. Let $c_H>0$ be the normalization for which
\begin{align}
c_H^2\int_\R\frac{1-\cos(\xi t)}{|\xi|^{1+2H}}\mathrm{d}\xi = \frac{1}{2}|t|^{2H}, \qquad\forall t\in\R.  \label{c_H}
\end{align}
Fixing a real orthonormal Meyer mother wavelet $\psi$, define the fractional mother wavelet $\Psi_H$ by 
\begin{equation}
\widehat{\Psi_H}(\xi)
=\sqrt{2\pi}
c_H|\xi|^{-H-1/2}\widehat\psi(\xi).
\label{eq:mst-fractional-wavelet}
\end{equation}
By the properties of Meyer wavelets, $\widehat\psi$ is smooth, compactly supported, and vanishes in a neighborhood of the origin. Therefore, $\Psi_H$ is a Schwartz function. Define
\begin{equation}
G_{j,k}^H(t)
\coloneqq
2^{-Hj}\bigl(\Psi_H(2^jt-k)-\Psi_H(-k)\bigr),
\qquad (j,k)\in\mathbb Z^2,\qquad t\in[0,1].
\label{eq:mst-kernels}
\end{equation}
An fBm with Hurst exponent $H$ on $[0, 1]$ admits the following \textit{all-scale white-noise} representation:
\begin{equation}
B^H(t)=\sum_{(j, k)\in\Z\times\Z}\e_{j,k}G_{j,k}^H(t),
\qquad \e_{j,k}\stackrel{\mathrm{iid}}{\sim}N(0,1), \qquad t\in[0,1].
\label{eq:mst-series-complete}
\end{equation}
This result was formally stated in \cite[Equation~(1.8)]{MeyerSellanTaqqu1999}; the rigorous version proved therein takes a slightly different form, involving a finite-scale white-noise representation together with an infrared correction. For completeness, we prove \eqref{eq:mst-series-complete}.  
\begin{lemma}\label{lm:allscale}
The series in \eqref{eq:mst-series-complete} converges uniformly on $[0, 1]$ a.s. to an fBm with Hurst exponent $H$ on $[0, 1]$.
\end{lemma}

The proof of \Cref{lm:allscale}, along with the results in the subsequent sections, relies on the following estimates. Fix $M\geq 4$ and define 
\begin{equation}\label{notation} 
\begin{aligned}
A_j \coloneqq\begin{cases}
2^{-Hj} & j\geq 0\\
2^{(1-H)j} & j<0
\end{cases}
,\quad N_j \coloneqq\begin{cases}
2^j & j\geq 0\\
1 & j<0
\end{cases}
,\quad \rho_{j,k}\coloneqq\begin{cases}
(1+\mathrm{dist}(k, [0, 2^j]))^{-M}& j\geq 0\\
(1+|k|)^{-M}& j<0
\end{cases}
\end{aligned}.
\end{equation}

\begin{lemma} 
\label{lem:mst-complete-envelopes}
For $H\in(0,1)$, the following estimates hold uniformly for all $j\in\Z$, and $s, t\in [0, 1]$:  
\begin{align}
\sum_{k\in\Z}\rho_{j,k}&\lesssim N_j\label{mst1}\\
\left\|\sum_{k\in\Z}|G_{j, k}^H|\right\|_\infty&\lesssim_H A_j\label{mst2}\\
\sum_{k\in\Z}\frac{|G^H_{j,k}(t)-G^H_{j,k}(s)|^2}{\rho_{j, k}}&\lesssim_H A_j^2\begin{cases}\min\{1, 2^{2j}|t-s|^2\}& j\geq 0\\|t-s|^2& j<0\end{cases}\label{mst3}
\end{align}
\end{lemma}

\begin{proof}
The bound \eqref{mst1} follows directly from the definition of $\rho_{j,k}$.
To establish \eqref{mst2}, for $j\geq 0$, take $x=2^jt\in[0,2^j]$. The Schwartz decay of $\Psi_H$ ensures the periodized absolute sum $\sum_k |\Psi_H(x-k)|$ uniformly bounded in $x$, and the same is true for $\sum_k |\Psi_H(-k)|$. This gives \eqref{mst2} after multiplication by $A_j$. For $j<0$, $2^j t<1/2$ for all $t\in [0, 1]$. By the mean-value theorem, for any $t\in [0, 1]$, 
\begin{align*}
|G_{j,k}^H(t)| = 2^{-Hj}\bigl|\Psi_H(2^jt-k)-\Psi_H(-k)\bigr| \leq 2^{(1-H)j}\sup_{x\in [-k, -k+1/2]}|\Psi'_H(x)|\lesssim A_j(|k|+1/2)^{-M},
\end{align*}
where the last step holds because $\Psi_H$ is Schwartz. Summing over $k$ yields the desired result. 

To establish \eqref{mst3}, for $j\geq 0$, we consider two different approaches. First, for $0\leq k\leq 2^j$, $\rho_{j,k}=1$, and the sum in this regime is
\begin{align*}
\sum_{0\leq k\leq 2^j}|G^H_{j,k}(t)-G^H_{j,k}(s)|^2 &= A_j^2\sum_{0\leq k\leq 2^j}|\Psi_H(2^jt-k)-\Psi_H(2^js-k)|^2\\
&\lesssim A_j^2\left(\sum_{k\in\Z}\Psi^2_H(2^jt-k)+\sum_{k\in\Z}\Psi^2_H(2^js-k)\right)\\
&\lesssim_H A_j^2, 
\end{align*}
where the last step follows from the finite periodized sum established earlier and is therefore uniform for all $s, t\in [0, 1]$. 
For $k\notin [0, 2^j]$ and $t\in [0, 1]$, $|2^jt-k|\geq\mathrm{dist}(k, [0, 2^j])$. The Schwartz property allows the decay of $|G^H_{j,k}(t)-G^H_{j,k}(s)|^2\lesssim A_j^2(1+\mathrm{dist}(k, [0, 2^j]))^{-2M}$, which remains summable over $k$ when divided by $\rho_{j,k}$. Combining the two sums yields the first part of the upper bound. 

Alternatively, we can use the mean-value theorem to obtain
\begin{align}
|G^H_{j,k}(t)-G^H_{j,k}(s)|= A_j\cdot 2^{j}|t-s|\cdot |\Psi'_H(2^j\xi(s, t, j, k)-k)|, \label{>0} 
\end{align}
where $\xi(s, t, j, k)$ is some number between $s$ and $t$. Applying a similar argument to $\Psi_H'$ as in the first approach by separating the cases $0\leq k\leq 2^j$ and $k\notin [0, 2^j]$ yields the second part of the upper bound. 
 
For $j<0$, the difference in the numerator is local, so only the second approach is effective. In fact, further bounding \eqref{>0} in this case yields 
\begin{align*}
|G^H_{j,k}(t)-G^H_{j,k}(s)|\leq A_j\cdot |t-s|\cdot \sup_{x\in [-k, -k+1/2]}|\Psi'_H(x)|\lesssim A_j|t-s|(1+|k|)^{-M}.
\end{align*}
Squaring both sides and dividing by $\rho_{j,k}$ and summing over $k$ yields the desired bound. 
\end{proof}

We now give the proof of \Cref{lm:allscale}. 
\begin{proof}[Proof of \Cref{lm:allscale}]
Let
\[
\mathcal H_{\mathbb R}
\coloneqq
\bigl\{h\in L^2(\mathbb R;\mathbb C):h(-\xi)=\overline{h(\xi)}\ \text{for a.e. }\xi\bigr\}
\]
be equipped with the real inner product
\[
\langle f,g\rangle_{\mathcal H_{\mathbb R}}
\coloneqq
\operatorname{Re}\int_{\mathbb R}f(\xi)\overline{g(\xi)}\,\mathrm d\xi.
\]
The Fourier transform identifies real $L^2(\mathbb R)$ isometrically with $\mathcal H_{\mathbb R}$. In particular, the Fourier transforms of the real Meyer wavelets
$\psi_{j,k}(t)\coloneqq2^{j/2}\psi(2^jt-k)$,
\[
e_{j,k}(\xi)\coloneqq\widehat{\psi}_{j,k}(\xi)=2^{-j/2}e^{-ik2^{-j}\xi}\widehat\psi(2^{-j}\xi),
\]
form an orthonormal basis of $\mathcal H_{\mathbb R}$ as a real Hilbert space. Let $W$ be the isonormal Gaussian process over $\mathcal H_{\mathbb R}$ determined by
$W(e_{j,k})=\e_{j,-k}$. For $t\in[0,1]$, set
\begin{align*}
h_t(\xi)\coloneqq \frac{c_H(e^{it\xi}-1)}{|\xi|^{1/2+H}}.
\end{align*}
Since $|h_t|^2=\mathcal O(|\xi|^{-(1+2H)})$ as $|\xi|\to\infty$ and $|h_t|^2=\mathcal O(|\xi|^{1-2H})$ as $|\xi|\to 0$, and $h_t(-\xi) = \overline{h_t(\xi)}$, $h_t\in\mathcal H_{\R}$ for every $H\in (0, 1)$. Consequently, $W(h_t)$ is a centered Gaussian process with the $L^2(\Omega)$-representation ($\Omega$ is the underlying probability space) for every $t\in [0, 1]$:   
\begin{align}
W(h_t)=\sum_{j,k}\langle h_t,e_{j,-k}\rangle_{\mathcal H_{\mathbb R}}\e_{j,k} =  \sum_{j,k}G^H_{j,k}(t)\e_{j,k},\label{hgyr}
\end{align}
where the second step follows by noting
\begin{align*}
\langle h_t,e_{j,-k}\rangle_{\mathcal H_{\mathbb R}}
&=\operatorname{Re}\left[2^{-j/2}\int_\mathbb R
\frac{c_H(e^{it\xi}-1)}{|\xi|^{1/2+H}}
 e^{-ik2^{-j}\xi}\overline{\widehat\psi(2^{-j}\xi)}\,\mathrm d\xi\right]\\
&=2^{-jH}\left(\Psi_H(2^jt-k)-\Psi_H(-k)\right)
=G_{j,k}^H(t),
\end{align*}
where the middle identity follows from the change of variables $\eta=2^{-j}\xi$ and \eqref{eq:mst-fractional-wavelet}. Moreover, the covariance of $W(h_t)$ can be computed as
\begin{align*}
\E[W(h_s)W(h_t)]
&=\langle h_s,h_t\rangle_{\mathcal H_{\mathbb R}}\\
&=c_H^2\int_{\mathbb R}
\frac{(1-\cos(s\xi))+(1-\cos(t\xi))-(1-\cos((t-s)\xi))}{|\xi|^{1+2H}}\,\mathrm d\xi\\
&=\frac12\left(s^{2H}+t^{2H}-|t-s|^{2H}\right)
\end{align*}
by \eqref{c_H}. Thus its finite-dimensional distributions are those of fBm with Hurst exponent $H$. 

To establish the uniform convergence of the series in \eqref{hgyr}, by the Gaussian tail estimates and the Borel--Cantelli lemma, $|\e_{j,k}|\leq C(\omega)\sqrt{\log(2+ |j|+ |k|)}$ for all $j, k\in\Z$ and some realization-dependent constant $C(\omega)<\infty$ a.s. In this case, 
\begin{align}
\left\|\sum_{k\in\Z}\e_{j,k}G_{j,k}^H\right\|_\infty&\leq C(\omega)\left\|\sum_{k\in\Z}\sqrt{\log(2+ |j|+ |k|)}\cdot |G_{j,k}^H|\right\|_\infty\lesssim_H C(\omega)A_j\sqrt{|j|+1},\label{feq}
\end{align}
where the last step follows from a similar estimate as in the proof of \eqref{mst2}. Since $\sum_j A_j\sqrt{|j|+1}<\infty$, the series converges uniformly a.s. to a continuous process, which agrees with the $L^2(\Omega)$-limit $W(h_t)$ a.s. for every fixed $t\in[0,1]$. Hence, this continuous process has the same finite-marginals as $W(h_t)$. This completes the proof.
\end{proof}

In the subsequent discussion, we consider the MST representation truncated at a prescribed threshold. The next lemma, which can be viewed as a discretization consequence of Bernstein's theorem applied to the MST representation, is needed for a boosting procedure in \Cref{subsec:cb}. 

\begin{lemma}\label{lem:sampling}
Given $J\in\Z$ and $\mathcal I\subseteq\{(j, k)\in\Z\times\Z: j\leq J\}$ finite, there is an equispaced grid $\mathcal T_J\subset[0,1]$ with $|\mathcal T_J|\lesssim_H 2^J+1$ such that every function of the form $g(t) = \sum_{(j,k)\in\mathcal I} c_{j,k}G^H_{j,k}(t), t\in [0, 1]$ with arbitrary real coefficients $(c_{j,k})_{(j,k)\in\mathcal I}$, satisfies
\begin{align}
\|g\|_\infty\leq 2\max_{t\in\mathcal T_J}|g(t)|. \label{eq:wavelet-sampling}
\end{align}
\end{lemma}
\begin{proof}
Since $\widehat{\Psi_H}$ is compactly supported, there exists $\Omega_H<\infty$ such that every function $G^H_{j,k}$, viewed on $\R$, has Fourier support contained in $[-\Omega_H2^j, \Omega_H2^j]$. Since $\mathcal I\subseteq\{(j,k): j\leq J\}$, every retained $G^H_{j,k}$ with $(j,k)\in\mathcal I$ has Fourier support contained in $[-\Omega_H2^J,\Omega_H2^J]$, and hence so does $g$, regardless of the coefficients $c_{j,k}$ or the particular subset $\mathcal I$ chosen.

Choose a Schwartz function $\varphi$ whose Fourier transform satisfies
$\widehat{\varphi}(\xi)=\frac{1}{\sqrt{2\pi}}$ for $|\xi|\le \Omega_H$, and set $\varphi_J(x)\coloneqq 2^J\phi(2^Jx)$.
Then $\widehat{\varphi_J}(\xi)=
\widehat{\varphi}(2^{-J}\xi)
=\frac{1}{\sqrt{2\pi}}
$ on the Fourier support of $g$.  With the MST convention
$\widehat f(\xi)=
\frac{1}{\sqrt{2\pi}}
\int_{\mathbb R}f(t)e^{-it\xi}\,\mathrm{d} t$. 
As a result, $g=g*\varphi_J$. Young's inequality gives
\begin{align}
\|g'\|_\infty
\leq \|g\|_\infty\|\varphi_J'\|_1
=2^J\|\varphi'\|_1\|g\|_\infty
\leq C_H2^J\|g\|_\infty.
\end{align}
Take $\mathcal T_J\subset[0,1]$ to be an equally spaced grid of mesh at most $(2C_H2^J)^{-1}$ containing both endpoints $0,1$; then $|\mathcal T_J|\lesssim_H 2^J+1$. If $t_\star$ maximizes $|g|$ on $[0,1]$ and $t\in\mathcal T_J$ is a nearest grid point to $t_\star$, then
\begin{align}
|g(t_\star)|\leq |g(t)|+|t-t_\star|\|g'\|_\infty
\leq |g(t)|+\frac12\|g\|_\infty,
\end{align}
which, upon rearranging and using $\|g\|_\infty=|g(t_\star)|$, proves \eqref{eq:wavelet-sampling} and the asserted cardinality bound.
\end{proof}

\subsection{Two balancing results}\label{subsec:vb}

Balancing the retained functions $f_{i, \ret}$ in \eqref{retain} requires the following vector balancing results. The first is an average-case result for isotropic Gaussian vectors in the regime $\omega(1)=m=o(n)$.
\begin{lemma}[Balancing Gaussian vectors]
\label{lem:tmr}
Suppose $x_i\sim N(0,I_m)$ are independent. If $\omega(1)=m=o(n)$, then with probability tending to one, 
\begin{align}
\disc(x_1, \ldots, x_n)\leq \sqrt{\pi n}2^{-n/m}. \label{gaussian-b}
\end{align}
\end{lemma}
This follows from \cite[Theorem~1]{turner2020balancing} by taking its fixed constant $\gamma=\sqrt2>1$; the same theorem also gives the matching lower scale. Thus, this result is sharp for balancing isotropic Gaussian vectors in the diverging sublinear regime $\omega(1)=m=o(n)$ and can be directly applied to the unweighted retained coefficients. 

When $m=\Omega(n)$, the bound in \eqref{gaussian-b} is no longer covered by the existing theory. Even if it were to remain valid, the resulting bound may no longer be effective. In this regime, we use an algorithmic balancing result for deterministic vectors with uniformly bounded norms. Rather than stating the balancing bound in a form as in \eqref{gaussian-b}, we give an equivalent subgaussian formulation that is more convenient for our subsequent applications.
\begin{lemma}[Algorithmic balancing]\label{lem:gs-subgaussian}
Given $u_1,\ldots,u_n\in\mathbb R^m$ with $\|u_i\|_2\leq\zeta$, there is a randomized polynomial-time algorithm returning $\sigma\in\{\pm1\}^n$ such that $\sum_{i=1}^n\sigma_i u_i$ is $(\sqrt{40}\,\zeta)$-subgaussian. Consequently, for every symmetric convex body $\mathcal K\subseteq\mathbb R^m$ of Gaussian measure at least $1/2$, there is an absolute constant $c>0$ such that
\[
\mathbb P_{\mathrm{alg}}\left\{\sum_{i=1}^n\sigma_i u_i\in (c\zeta)\cdot\mathcal K\right\}\geq\frac12.
\]
\end{lemma}
\begin{proof}
Applying the Gram--Schmidt walk algorithm in \cite{bansal2018gram} to the vectors $u_i/\zeta$ with initial fractional coloring $0$ yields a signing $\sigma\in\{\pm1\}^n$ such that $\sum_{i=1}^n \sigma_i(u_i/\zeta)$ is $\sqrt{40}$-subgaussian  \cite[Theorem~1.3]{bansal2018gram}. The convex-body conclusion follows from the subgaussian form of Banaszczyk's theorem due to \citet{dadush2019towards}, equivalently \cite[Theorem~1.2]{bansal2018gram}, for symmetric convex bodies.
\end{proof}

We apply \Cref{lem:gs-subgaussian} to the weighted retained coefficients to obtain an effective balancing of the retained functions. This reduces the problem to controlling a family of functionals of the resulting subgaussian vector indexed by time. In this setting, \Cref{lem:gs-subgaussian} yields the following corollary, which is the key ingredient in the analysis of \Cref{subsec:cb}.
\begin{corollary}\label{lm:ds}
Let $F:\mathbb R^m\to C([0,1])$ be linear and define
\begin{equation}
C_F\coloneqq\mathbb E_{z\sim N(0,I_m)}\bigl[\|F(z)\|_\infty\bigr].
\label{myCF}
\end{equation}
For any $u_1,\ldots,u_n\in\mathbb R^m$ with $\|u_i\|_2\leq\zeta$, there is a randomized polynomial-time algorithm returning $\sigma\in\{\pm1\}^n$ such that, with probability at least $1/2$,
\[
\left\|F\left(\sum_{i=1}^n\sigma_i u_i\right)\right\|_\infty
\leq (c\zeta)\cdot C_F
\]
for an absolute constant $c>0$.
\end{corollary}
\begin{proof}
Let $\mathcal Q\subseteq\R^m$ be a centered (closed) Euclidean ball with Gaussian measure $3/4$. Define
\begin{align*}
\mathcal K =\bigl\{z\in\R^m:\|F(z)\|_\infty\leq 4C_F\bigr\}\cap\mathcal Q, 
\end{align*}
which is symmetric, convex, and compact. Markov's inequality gives $\mathbb P\{\|F(z)\|_\infty\leq4C_F\}\geq3/4$, so the union bound yields Gaussian measure at least $1/2$ for $\mathcal K$ (i.e., $\mathcal K$ has nonempty interior). Applying \Cref{lem:gs-subgaussian} to $\mathcal K$ proves the claim after absorbing the factor $4$ into the absolute constant.
\end{proof}

Using the Riesz representation, one can further write $F(z) = \langle z, \theta(t)\rangle$ for some $\theta(t): [0, 1]\to\R^m$. With this, the quantity $C_F$ corresponds to the Gaussian complexity \cite[Section 7.5]{vershynin2018high} of $F$ in the dual space, i.e., $\{\theta(t): t\in [0, 1]\}\subset\R^m$. 

\subsection{The existential upper bound}
\label{subsec:eb} 

For $R\geq2$, define the retaining indices at $j$th scale by 
\begin{align*}
\mathcal I_j(R)\coloneqq
\begin{cases}
\{k\in\mathbb Z:-R\leq k\leq 2^j+R\},&j\geq0,\\
\{k\in\mathbb Z:|k|\leq R+1\},&j<0.
\end{cases}
\end{align*}
The retaining set $\mathcal I$ with resolution range between $K_1\leq K_2$ is defined by
\begin{align}
\mathcal I\coloneqq\mathcal I(K_1, K_2, R)\coloneqq\{(j, k): K_1\leq j\leq K_2, k\in \mathcal I_j(R)\}. \label{rset}
\end{align}
Under appropriate choices of $K_1$, $K_2$, and $R$ and using \Cref{lem:tmr}, we establish the following upper bound on $\disc(f_1,\ldots,f_n)$. 

\begin{proposition}
\label{prop:mst-complete-upper}
Let $f_1,\ldots,f_n$ be independent fBm sample paths with Hurst exponent $H\in(0,1)$. Then, with probability $1-o(1)$,
\[
\disc(f_1,\ldots,f_n)
\lesssim_H n^{1/2-H}(\log n)^{H+1/2}.
\]
\end{proposition}

\begin{proof}
Recall the MST wavelet representation of $f_i$ in \eqref{eq:mst-series-complete}:
\begin{align}
f_i(t)=\sum_{j,k}\varepsilon_{i,j,k}G_{j,k}^H(t),
\qquad
\e_{i,j,k}\stackrel{\mathrm{iid}}{\sim}N(0,1).
\end{align}
Set
\begin{equation}
K_1\coloneqq -\left\lfloor\frac{H}{1-H}\log_2\!\left(\frac{n}{\log n}\right)\right\rfloor,
\qquad
K_2\coloneqq\left\lfloor\log_2\!\left(\frac{n}{8\log n}\right)\right\rfloor,
\qquad
R\coloneqq\lceil n^{1/4}\rceil.
\label{eq:mst-complete-parameters}
\end{equation}
Since $|\mathcal I_j(R)|\leq 2^j+2R+1$ for $j\geq0$ and $|\mathcal I_j(R)|\leq 2R+3$ for $j<0$, 
\begin{equation}
n^{1/4}\leq R\leq m\coloneqq |\mathcal I|\leq C\bigl(2^{K_2+1}+(K_2-K_1+1)(2R+3)\bigr)
\leq\frac{n}{4\log n}+o\left(\frac{n}{\log n}\right)
\le
\frac{n}{3\log n}
\label{eq:mst-complete-dimension}
\end{equation}
for all large $n$.

The retained functions $f_{i, \ret}$ are defined in \eqref{retain} based on the chosen $\mathcal I$. For each $i$, let
\[
x_i\coloneqq (\varepsilon_{i,j,k})^\top_{(j,k)\in\mathcal I}\in\mathbb R^m.
\]
The vectors $x_i$ are i.i.d. $N(0,I_m)$. Since $\omega(1)=m=o(n)$, \Cref{lem:tmr} gives, with probability $1-o(1)$, there exists a signing $\sigma=\sigma(x_1,\ldots,x_n)$ such that
\begin{equation}
\left\|\sum_{i=1}^n\sigma_i x_i\right\|_\infty
\leq\sqrt{\pi n}2^{-n/m}\lesssim n^{-1},
\label{eq:mst-complete-vector-balance}
\end{equation}
where we used \eqref{eq:mst-complete-dimension} in the second inequality. 
Applying H\"older's inequality and \Cref{lem:mst-complete-envelopes}, 
\begin{align*}
\left\|\sum_{i=1}^n\sigma_if_{i, \ret}\right\|_\infty = \left\|\sum_{(j, k)\in\mathcal I}\left(\sum_{i=1}^n\sigma_i\e_{i,j,k}\right)G^H_{j,k}\right\|_\infty&\leq \left\|\sum_{(j, k)\in\mathcal I}|G^H_{j,k}|\right\|_\infty\left\|\sum_{i=1}^n\sigma_i x_i\right\|_\infty\\
&\lesssim_H \left(\sum_{j=0}^{K_2}2^{-Hj}
+\sum_{j=K_1}^{-1}2^{j(1-H)}\right)\cdot n^{-1}\\
&\lesssim_H n^{-1}.  
\end{align*} 

To bound the omitted series, we condition on the retained coefficients and hence on $\sigma$. For each omitted index $(j,k)\notin\mathcal I$, define $\xi_{j,k}\coloneqq\sum_{i=1}^n\sigma_i\varepsilon_{i,j,k}$. The all-scale white noise property of the MST expansion implies that the random variables $\{\xi_{j,k}/\sqrt n\}_{(j,k)\notin\mathcal I}$ are conditionally i.i.d. $N(0,1)$ given the retained coefficients. A standard Gaussian tail estimate and a union bound over $\Z^2$ imply that, for some constant $C>0$, with conditional probability at least $1-\Oo(n^{-4})$, 
\begin{align}
|\xi_{j,k}|\leq C\sqrt{n(\log n + \log (2+|j|+|k|))}, \qquad\forall (j, k)\notin\mathcal I. \label{xi_1}
\end{align}
For instance, $C$ can be chosen so that the sum of the individual failure probabilities is bounded by
\[
C\sum_{j,k\in\mathbb Z}n^{-5}(2+|j|+|k|)^{-4}=\Oo(n^{-5}).
\]
Conditional on \eqref{xi_1}, the truncation error under $\sigma$ can be bounded by
\begin{align}
\left\|\sum_{i=1}^n\sigma_i(f_i-f_{i,\mathrm{ret}})\right\|_\infty
&\leq \left\|\sum_{j>K_2}\sum_{k\in\mathbb Z}\xi_{j,k}G_{j,k}^H\right\|_\infty
+\left\|\sum_{j<K_1}\sum_{k\in\mathbb Z}\xi_{j,k}G_{j,k}^H\right\|_\infty
+ \left\|\sum_{j=K_1}^{K_2}\sum_{k\notin \mathcal I_j(R)}\xi_{j,k}G_{j,k}^H\right\|_\infty\nonumber\\
&\lesssim_H \sqrt{n\log n}\,2^{-HK_2} + \sqrt{n\log n}\,2^{K_1(1-H)} + \sqrt{n\log n}\,R^{-6}.\label{remainder}
\end{align}
The first two bounds follow from a similar reasoning as in \eqref{feq}. 
For the last bound, outside $\mathcal I_j(R)$ both arguments of the two Schwartz functions are at distance at least $R$ from the relevant region. We use the spatial decay of the Schwartz function $\Psi_H$. Recall that, for $j\geq 0$ and $t\in[0,1]$, we have $2^jt\in[0,2^j]$. Hence, if $k\notin \mathcal I_j(R)$, then
\[
d_{j,k}:=\operatorname{dist}(k,[0,2^j])\geq R,
\]
and both $|2^jt-k|$ and $|k|$ are at least $d_{j,k}$. Since $\Psi_H$ is Schwartz, for every $M>0$,
\[
\|G^H_{j,k}\|_\infty
\lesssim_{H,M}2^{-Hj}(1+d_{j,k})^{-M},
\qquad j\geq0,\quad k\notin \mathcal I_j(R).
\]
For $j<0$, the mean-value argument used in the proof of the wavelet estimates similarly gives
\[
\|G^H_{j,k}\|_\infty
\lesssim_{H,M}2^{(1-H)j}(1+|k|)^{-M},
\qquad j<0,\quad k\notin \mathcal I_j(R).
\]
Thus, writing $A_j=2^{-Hj}$ for $j\geq0$ and
$A_j=2^{(1-H)j}$ for $j<0$, we obtain uniformly for
$K_1\leq j\leq K_2$,
\[
\|G^H_{j,k}\|_\infty
\lesssim_{H,M}A_j
\bigl(1+\operatorname{dist}(k, \mathcal I_j(0))\bigr)^{-M},
\qquad k\notin \mathcal I_j(R).
\]
Combining this estimate with the Gaussian coefficient bound in \eqref{xi_1}, and using $|j|=\mathcal O_H(\log n)$ throughout the retained range determined by~\eqref{eq:mst-complete-parameters}, gives
\[
\begin{aligned}
\sum_{k\notin \mathcal I_j(R)}
|\xi_{j,k}|\|G^H_{j,k}\|_\infty
&\lesssim_{H,M}
\sqrt n\,A_j
\sum_{d\geq R}
\sqrt{\log n+\log(2+|j|+d)}\,(1+d)^{-M}  \\
&\lesssim_H
\sqrt{n\log n}\,A_jR^{-6},
\end{aligned}
\]
where in the last step we choose $M$ sufficiently large and absorb the additional logarithmic growth into the polynomial Schwartz tail. The particular power $6$ is arbitrary; any sufficiently large fixed power would suffice. Finally,
\[
\sum_{j=K_1}^{K_2}A_j
\leq
\sum_{j\geq0}2^{-Hj}
+\sum_{j<0}2^{(1-H)j}
\lesssim_H1,
\]
so
\[
\left\|
\sum_{j=K_1}^{K_2}\sum_{k\notin \mathcal I_j(R)}
\xi_{j,k}G^H_{j,k}
\right\|_\infty
\lesssim_H \sqrt{n\log n}\,R^{-6},
\]
which is the third bound in \eqref{remainder}. The remaining sum over $j$ is geometric (across scales) and is thus finite. The parameter choices in \eqref{eq:mst-complete-parameters} imply
\[
2^{-HK_2}\asymp n^{-H}(\log n)^H,
\qquad
2^{K_1(1-H)}\asymp n^{-H}(\log n)^H,
\qquad
R^{-6}\leq n^{-H}
\]
for all large $n$. Thus, each of the three omitted contributions above is 
\[
\Oo\left(n^{1/2-H}(\log n)^{H+1/2}\right).
\]
Combining the retained and omitted parts proves the \Cref{prop:mst-complete-upper}.
\end{proof}

Under our choice of truncation, the truncation error dominates for all $H \in (0,1)$. Although $K_2$ can be adjusted slightly to balance the errors arising from the two phases, doing so only affects constant factors: it neither changes the order of $K_2$ nor improves the asymptotic rate. We therefore do not pursue this optimization here.

\subsection{The constructive upper bound}\label{subsec:cb} 

Proposition~\ref{prop:mst-complete-upper} establishes an upper bound by balancing the unweighted retained coefficients for $m=o(n)$. Extending this approach to higher dimensions requires algorithmic balancing techniques. Since these algorithms are worst-case guarantees whose performance depends on the vector norms rather than independence, we instead balance appropriately weighted coefficients that reflect the regularity structure of the underlying functions. 

Recall the retaining set $\mathcal I = \mathcal I(K_1, K_2, R)$ defined in \eqref{rset}. 
Given positive weights $w\coloneqq \{w_{j,k}\}$, define the linear mapping $F_w: \R^m\to C([0, 1])$ by 
\begin{align}
F_w(z)\coloneqq \left\langle z, \theta(t; w)\right\rangle, \qquad \theta(t; w)\coloneqq
(w^{-1}_{j,k}G_{j,k}^H(t))^\top_{(j,k)\in\mathcal I}\in\R^m.  \label{F_w}
\end{align}
Using this definition, we rewrite $f_{i, \ret}$ in \eqref{retain} as 
\begin{align}
f_{i, \ret} = F_{w}(u_i), \qquad u_i\coloneqq (w_{j ,k}\e_{i, j, k})^\top_{(j, k)\in\mathcal I}\in\R^m, \label{key1}
\end{align} 
and consider balancing the $u_i$ instead. Under an appropriate choice of $w$ and applying \Cref{lm:ds}, we establish the following upper bound.  
\begin{proposition}\label{prop:gs-upper}
Let $H\in(0,1)$ and let $f_1,\ldots,f_n$ be independent fBm sample paths with Hurst exponent $H$. There is a randomized polynomial-time algorithm in the finite retained-coefficient real-arithmetic input model which returns $\sigma\in\{\pm1\}^n$ such that, with probability $1-o(1)$ over the sample paths and the internal randomness of the algorithm,
\begin{equation}
\left\|\sum_{i=1}^n\sigma_i f_i\right\|_\infty
\lesssim_H
\begin{cases}
n^{1/2-H}\sqrt{\log n},&0<H<1/2,\\
(\log n)^{3/2},&H=1/2,\\
\sqrt{\log n},&1/2<H<1.
\end{cases}
\label{eq:all-regime-algorithm}
\end{equation}
\end{proposition}
We first provide a heuristic explanation of how to choose $w$ and then prove Proposition~\ref{prop:gs-upper}. 

\paragraph{Heuristics on the choice of $w$.}

A good choice of the weights $w$ should balance the competing objectives of keeping $\|u_i\|_2$ small and controlling $\|F_w\|$, the latter being crucial for bounding $C_{F_w}$ in \eqref{myCF}. For ease of analysis, we choose
\begin{align}
w_{j,k}=\sqrt{\lambda_j\rho_{j,k}},\label{sweight}
\end{align}
where $\rho_{j,k}$ is the soft-thresholding sequence defined in \eqref{notation}, which equalizes the weights of the coefficients at scale $k\in [0, 2^j]$ while gradually downweighting shifts away from the interval, and $\lambda_j$ are global scale parameters. Under such choice and using \Cref{lem:mst-complete-envelopes}, 
\begin{equation}\label{uF}
\begin{aligned}
\E[\|u_i\|_2^2] &=\E[\|u_1\|_2^2] = \sum_{(j, k)\in\mathcal I}\lambda_j\rho_{j,k}\lesssim \sum_{j=K_1}^{K_2}\lambda_j N_j,\\
\|F_w\|^2&\leq \sup_{t\in [0, 1]}\|\theta(t; w)\|_2^2 = \sup_{t\in [0, 1]}\sum_{(j, k)\in\mathcal I}
\frac{|G_{j,k}^H(t)|^2}{\lambda_j\rho_{j,k}}\\
&=\sup_{t\in [0, 1]}\sum_{(j, k)\in\mathcal I}
\frac{|G_{j,k}^H(t)-G_{j,k}^H(0)|^2}{\lambda_j\rho_{j,k}}\lesssim_H \sum_{j=K_1}^{K_2}\frac{A_j^2}{\lambda_j}.  
\end{aligned}
\end{equation}
To strike a balance between the two sums, we equate the summands on the same scale, yielding 
\begin{align}
\lambda_j = \frac{A_j}{\sqrt{N_j}}. \label{lambdah}
\end{align} 
\paragraph{Proof of Proposition~\ref{prop:gs-upper}.}
For the retaining set $\mathcal I = \mathcal I(K_1, K_2, R)$ in \eqref{rset}, set
\begin{equation}
K_1\coloneqq -\left\lfloor\frac{H}{1-H}\log_2n\right\rfloor,
\qquad
K_2\coloneqq\lfloor\log_2 n\rfloor,
\qquad
R\coloneqq n,
\label{eq:all-regime-parameters}
\end{equation}
Its cardinality is $m=\Oo(n\log n)$, so the finite-dimensional operations below have polynomial complexity.

We now fix $w$ in \eqref{sweight} with $\lambda_j$ chosen in \eqref{lambdah}. According to \eqref{uF}, up to a multiplicative constant depending only on $H$, both $\E[\|u_i\|_2^2]$ and $\|F_w\|^2$ are bounded by  
\begin{equation}\label{Qbdd}
\begin{aligned}
Q\coloneqq \sum_{j=K_1}^{K_2}\lambda_j N_j = \sum_{j=K_1}^{K_2}A_j \sqrt{N_j}&=\sum_{j=K_1}^{-1}2^{(1-H)j}+\sum_{j=0}^{K_2}2^{(-H+1/2)j}\\
&\lesssim_H\begin{cases}
n^{1/2-H},&0<H<1/2,\\
\log n,&H=1/2,\\
1,&1/2<H<1.
\end{cases}
\end{aligned}
\end{equation}
Write $\alpha_{j,k}=w_{j,k}^2=\lambda_j\rho_{j,k}$. Then $0<\alpha_{j,k}\leq1$ and
$T\coloneqq\sum_{(j,k)\in\mathcal I}\alpha_{j,k}\lesssim_H Q$. The weighted Laurent--Massart inequality gives
\[
\mathbb P\left\{\|u_i\|_2^2>T+2\sqrt{Tt}+2t\right\}\leq e^{-t}.
\]
Taking $t=4\log n$ and applying a union bound over $i$ yields, with probability at least $1-\Oo(n^{-3})$,
\begin{align}
\max_i\|u_i\|_2\lesssim_H\sqrt{Q+\log n}.
\label{ubdd}
\end{align}

Conditional on \eqref{ubdd}, recalling the linear functional representation of $f_{i, \ret}(t)$ in \eqref{F_w}-\eqref{key1} and applying \Cref{lm:ds}, there exists a randomized polynomial-time algorithm that returns $\sigma$ such that, with probability at least $1/2$, 
\begin{align}\label{ksjg}
\left\|\sum_{i=1}^n\sigma_if_{i, \ret}\right\|_\infty\lesssim_H \sqrt{Q+ \log n}\cdot C_{F_w}, \qquad C_{F_w}\coloneqq \E_{z\sim N(0, I_m)}[\|F_w(z)\|_\infty],
\end{align}
where $F_w$ is defined in \eqref{F_w}. To further bound $C_{F_w}$, note that $F_w(z)$ is a Gaussian process indexed by $t\in [0, 1]$ started from zero. Define the square of the $L^2$-metric induced by $F_w$ on $[0, 1]$\footnote{This is equivalent to the canonical metric induced by the subgaussian norm for Gaussian processes.}: 
\begin{align}
d^2_{F_w}(s, t) \coloneqq \E[|F_w(z)(s)-F_w(z)(t)|^2]= \|\theta(s; w)-\theta(t; w)\|^2_2 = \sum_{(j, k)\in\mathcal I}
\frac{|G_{j,k}^H(s)-G_{j,k}^H(t)|^2}{\lambda_j\rho_{j,k}}, 
\end{align}
and let $\mathcal N([0, 1], d_{F_w}, \e)$ denote the $\e$-covering number of $[0, 1]$ under $d_{F_w}$. By Dudley's inequality \cite[Remark 8.15]{vershynin2018high},   
\begin{align}
C_{F_w}\lesssim_H\int_0^{\mathrm{diam}([0, 1], d_{F_w})}\sqrt{\log \mathcal N([0, 1], d_{F_w}, \e)} \ \mathrm d\e. \label{dudley}
\end{align}
By the triangle inequality and \eqref{uF}, $\mathrm{diam}([0,1],d_{F_w})\leq 2\|F\|\lesssim_H\sqrt Q$. In particular, there exists a constant $C_H>0$ such that
\begin{align}
\mathrm{diam}([0, 1], d_{F_w})\leq C_H\sqrt{Q}. \label{diam}
\end{align} 
We now use the quadratic estimates in \eqref{mst3} to bound $d^2_{F_w}(s, t)$, from which we obtain the bounds on the covering number. Denoting $r = |t-s|$, using \eqref{mst3}, and recalling the choice of $w$ in \eqref{uF}, we obtain
\begin{align}
d^2_{F_w}(s, t)&\lesssim_H r^2\sum_{j=K_1}^{-1}2^{(1-H)j} + \sum_{j=0}^{K_2} 2^{(-H+1/2)j}\min\{1, 2^{2j}r^2\}. \label{tom}
\end{align}

For $H\leq 1/2$, bounding $2^{2j}$ by $2^{2K_2}$, dropping $1$ in the minimum, and recalling the definition of $Q$ in \eqref{Qbdd}, $d_{F_w}(s, t) \lesssim_H 2^{K_2} \sqrt{Q}r$. Consequently, 
\begin{align}
\mathcal N([0,1],d_{F_w},r)\leq1+\frac{C_H2^{K_2}\sqrt Q}{r},\qquad r>0. \label{cn1}
\end{align}
Substituting \eqref{diam} and \eqref{cn1} into \eqref{dudley}, and then scaling $r=C_H\sqrt Q\,u$, yields
\begin{align*}
C_{F_w}
&\lesssim_H\int_0^{C_H\sqrt Q}
\sqrt{\log\left(1+C_H2^{K_2}\sqrt Q/r\right)}\,\mathrm dr = C_H\sqrt Q\int_0^1\sqrt{\log\left(1+2^{K_2}/u\right)}\,\mathrm du
\lesssim_H\sqrt{QK_2}.
\end{align*}

For $H>1/2$, denoting $\tau = H-1/2>0$, further computing the upper bound in \eqref{tom} by separating the sum based on the minimum,  
\begin{align*}
d^2_{F_w}(s, t)\lesssim_H \sum_{j=0}^{\lfloor\log_2(1/r)\rfloor} 2^{(2-\tau)j}r^2 + \sum_{j=\lfloor\log_2(1/r)\rfloor+1}^{K_2} 2^{-\tau j}\lesssim_H r^{\tau}, 
\end{align*}
which gives
\begin{align}
\mathcal N([0,1],d_{F_w},r)\leq1+C_Hr^{-2/\tau},\qquad r>0. \label{cn2}
\end{align}
Since $Q\lesssim_H1$ in this regime, substituting \eqref{diam} and \eqref{cn2} into \eqref{dudley} gives
\begin{align*}
C_{F_w}\lesssim_H\int_0^{C_H}\sqrt{\log(1+C_Hr^{-2/\tau})}\,\mathrm dr\lesssim_H1.
\end{align*}
Putting the above estimates together,  
\begin{align}\label{jsgf}
\left\|\sum_{i=1}^n\sigma_if_{i, \ret}\right\|_\infty\lesssim_H \begin{cases}
n^{1/2-H}\sqrt{\log n}&0<H<1/2\\
(\log n)^{3/2}&H=1/2\\
\sqrt{\log n}&1/2<H<1.
\end{cases}
\end{align}

To show that a signing satisfying \eqref{jsgf} can be found in polynomial time, note that a single execution of the Gram--Schmidt walk on the weighted retained coefficients produces such a signing with probability at least $1/2$. Repeating the algorithm independently $\Oo(\log n)$ times boosts the success probability to at least $1-\Oo(n^{-1})$. Among the sampled sign vectors, we then select the one minimizing $\|\sum_{i=1}^n \sigma_i f_{i,\ret}\|_\infty$. By \Cref{lem:sampling}, this $L^\infty$ norm can be evaluated on an equispaced discretization of size $\Oo(2^{K_2})=\Oo(n)$, incurring only a multiplicative factor of $2$ in the discrepancy bound. Consequently, the boosting and discretization steps together increase the computational cost by only a polynomial factor for every fixed $H$.

It remains to control the omitted coefficients for the signing selected by the boosting procedure. That signing is measurable with respect to the retained coefficients and the internal randomness of the algorithm. Hence, conditional on this information, for every omitted index $(j,k)\notin\mathcal I$, $\xi_{j,k}\coloneqq\sum_{i=1}^n\sigma_i\e_{i,j,k}\sim N(0,n)$, and the family $(\xi_{j,k})_{(j,k)\notin\mathcal I}$ is conditionally independent. The same Gaussian union bound and the three envelope estimates used in \eqref{remainder} therefore give, with conditional probability $1-\Oo(n^{-4})$,
\begin{align*}
\left\|\sum_{j>K_2}\sum_k\xi_{j,k}G_{j,k}^H\right\|_\infty
&\lesssim_H\sqrt{n\log n}\,2^{-HK_2},\\
\left\|\sum_{j<K_1}\sum_k\xi_{j,k}G_{j,k}^H\right\|_\infty
&\lesssim_H\sqrt{n\log n}\,2^{K_1(1-H)},\\
\left\|\sum_{j=K_1}^{K_2}\sum_{k\notin \mathcal I_j(R)}\xi_{j,k}G_{j,k}^H\right\|_\infty
&\lesssim_H\sqrt{n\log n}\,R^{-2}.
\end{align*}
With the choices in \eqref{eq:all-regime-parameters}, this is
\begin{align}
\Oo\bigl(n^{1/2-H}\sqrt{\log n}+n^{-3/2}\sqrt{\log n}\bigr).
\label{jsgl}
\end{align}
This matches the retained estimate for $H<1/2$ and is absorbed for $H\geq1/2$. Combining \eqref{jsgf} and \eqref{jsgl}, together with the coefficient-concentration and boosting probabilities, proves the proposition.  
  
\subsection{The lower bound}\label{subsec:lb} 

Our proof for the lower bound does not rely on any specific expansion of fBm and instead follows from a first-moment computation. The key observation is that $\frac{1}{\sqrt{n}}\sum_{i=1}^n\sigma_i f_i$ has the same distribution as $f_1$ for every fixed $\sigma$, and is therefore itself an fBm. This invariance property is a general property of Gaussian processes. The result then follows from the classical small-ball probability estimates of fBm.

For $a>0$, define 
\begin{align}\label{Zn}
&Z_n(a)\coloneqq\sum_{\sigma\in\{\pm 1\}^n}\mathbf 1\{\mathcal A(\sigma)\}, &\mathcal A(\sigma)\coloneqq \left\{\left\|\sum_{i=1}^n \sigma_i f_i\right\|_{\infty}\leq a\right\}. 
\end{align}
Since the $f_i$ are i.i.d. Gaussian processes (i.e., $B^H$), for any fixed signing $\sigma$, 
\begin{equation}
\frac{1}{\sqrt{n}}\sum_{i=1}^n\sigma_i f_i
\ \stackrel{d}{=}\ B^H.
\label{eq:mst-complete-fixed-sign-law}
\end{equation}
Consequently, by Markov's inequality,
\begin{align}
\mathbb P\{\disc(f_1,\ldots,f_n)\leq a\}\leq \mathbb E [Z_n(a)]=2^n\,
\mathbb P\left\{\|B^H\|_\infty\leq\frac{a}{\sqrt n}\right\}.
\label{eq:mst-complete-first-moment}
\end{align}
Then an fBm small-ball estimate \cite[Theorem 4.6]{LiShao2001} gives constants $C_1=C_1(H)>0$ and $\e<1$ such that
$\Pp\{\|B^H\|_\infty\leq\varepsilon\}
\leq\exp\{-C_1\e^{-1/H}\}$. 
Taking $a=C_Hn^{1/2-H}$ gives $a/\sqrt n=C_Hn^{-H}=o(1)$ and, for all sufficiently large $n$,
\[
\E [Z_n(a)]
\leq
\exp\left\{\bigl(\log2-C_1C_H^{-1/H}\bigr)n\right\}.
\]
Choose $C_H>0$ sufficiently small that $C_1C_H^{-1/H}>\log2$. Then the exponent is at most $-\gamma_Hn$ for some $\gamma_H>0$, and \eqref{eq:mst-complete-first-moment} proves the claim.

For $H = 1/2$, let $B$ be a standard Brownian motion. The constant $C_1$ can be made precise by using the exact asymptotics for all small $\e$: 
\begin{align}\label{exact}
\Pp\left\{\|B\|_{\infty}\le \varepsilon\right\}=
\frac{4}{\pi}
\sum_{j=0}^{\infty}
\frac{(-1)^j}{2j+1}
\exp\left\{
-\frac{(2j+1)^2\pi^2}{8\e^2}
\right\}= \left(\frac{4}{\pi} + o(1)\right)\exp\left\{-\frac{\pi^2}{8\e^2}\right\}.
\end{align}
The distribution of $\sup_{s\le t}|B(s)|$ is given, in an equivalent image-series form, in \cite[Part II, Ch.~3, Formula~1.1.4]{BorodinSalminen2002}; the displayed eigenfunction series is the equivalent Dirichlet heat-kernel expansion on $(-\varepsilon,\varepsilon)$. Therefore, for every fixed $\delta\in(0, a_0)$, taking $c=a_0-\delta$ gives $\E[Z_n(c)]=o(1)$.

%% file: critical-regime.tex
\section{Critical regime}\label{sec:(ii)}

\subsection{A second-moment bound}\label{sec:1/2}

We now consider the regime $H=1/2$ and provide a refined upper bound using a second-moment computation. It is enough to prove the result under $a_n\leq n^{1/3}$. Indeed, set
\[
\widetilde a_n\coloneqq\min\{a_n,n^{1/3}\}.
\]
Monotonicity gives
$\mathbb P\{\disc(f_1,\ldots,f_n)\leq a_n\}\geq
\mathbb P\{\disc(f_1,\ldots,f_n)\leq\widetilde a_n\}$.
Whenever truncation occurs,
\[
0\leq
\sqrt{1-\frac{\beta}{a_n^2}}
-\sqrt{1-\frac{\beta}{\widetilde a_n^2}}
\leq 1-\sqrt{1-\beta n^{-2/3}}
=\Oo(n^{-2/3}).
\]
Thus the desired lower bound for $\widetilde a_n$ implies the stated bound for $a_n$ after absorbing this deterministic error into $o(1)$. We henceforth relabel $\widetilde a_n$ as $a_n$, so that $a_n^2=o(n)$. The following lemmas are needed for the subsequent analysis. 

\begin{lemma}\label{lm:corB}
Let $B(t) = (B_1(t), B_2(t))^\top$ be a planar Brownian motion started from the origin with constant correlation $\rho\in (-1, 1)$. Define 
\begin{align}
&\mathcal R(\rho)\coloneqq T(\rho)^{-1}([-1 ,1]^2)\subset\R^2, &T(\rho) = \begin{pmatrix} 1 & 0 \\ \rho & \sqrt{1-\rho^2}\label{Rrho}
\end{pmatrix}. 
\end{align}
Let $0<\lambda_1(\rho) < \lambda_2(\rho)\leq \dots$ be the Dirichlet eigenvalues of the negative Laplacian $(-\Delta)$ on $\mathcal R(\rho)$, and let $\{\phi_k(\cdot; \rho)\}_{k=1}^\infty$ be the corresponding $L^2$-normalized eigenfunctions. The following results hold:
\begin{enumerate} 
\item [(i)] Fixing $\delta>0$, the small-ball probability asymptotics hold uniformly for $|\rho|\leq 1-\delta$ as $\e\to 0$:  
\begin{align}
\Pp\left\{B(t)\in [-\e, \e]^2, \forall t\in [0, 1]\right\} = C_1(\rho)\exp\left\{-\frac{\lambda_1(\rho)}{2\varepsilon^2}\right\} (1 + o(1))\label{goal}
\end{align}
where the constant $C_1(\rho)$ is strictly positive and given by
\begin{align}
C_1(\rho)= \phi_1(0; \rho) \int_{\mathcal R(\rho)} \phi_1(x; \rho) \mathrm{d} x, \label{myc1}
\end{align}
and the $o(1)$ term depends on $\delta$. 
\item [(ii)] The principal eigenvalue $\lambda_1(\rho)$ is real-analytic, even, concave for $\rho\in(-1,1)$, and satisfies the following lower bound:
\begin{align}
\lambda_1(\rho)\geq\frac{\pi^2}{4}(1+\sqrt{1-\rho^2}). \label{newguy}
\end{align}
Moreover, $\lambda_1'(0)=0$, $\lambda_1'(\rho)<0$ for $0<\rho<1$, $\lambda_1'(\rho)>0$ for $-1<\rho<0$, and $-\lambda_1''(0)\in [1.965,1.966]$. 
\item [(iii)] $C_1(\rho)$ is continuous for $\rho\in(-1, 1)$. 
\end{enumerate}
\end{lemma}
\begin{proof}
\Cref{sec:corB}.  
\end{proof}

\begin{lemma}\label{lem:binom}
For $|k|\leq n^{3/5}$,
\begin{align}
\frac{1}{2^n}{n\choose{n/2+k}}=\sqrt{\frac{2}{\pi n}}e^{-2k^2/n}\exp\left\{\mathcal O\left(\frac{1}{n}+\frac{k^4}{n^3}\right)\right\} = \sqrt{\frac{2}{\pi n}}e^{-2k^2/n}(1+o(1)).
\end{align}
\end{lemma}
\begin{proof}
This lemma follows from a direct application of Stirling's formula and Taylor expansion, and we provide it for completeness. Write $p=1/2+k/n$. By Stirling's formula,
\begin{align}
\frac{1}{2^n}{n\choose{n/2+k}}=\sqrt{\frac{1}{2\pi np(1-p)}}\exp\left\{-n(\log 2-h(p))\right\}\left(1+\Oo(1/n)\right),\label{binomial-stirling}
\end{align}
where $h$ is the binary entropy function in \eqref{bi-entropy}. 
Since $h$ is symmetric with respect to $p=1/2$, its Taylor expansion at $p=1/2$ contains even terms only: 
\begin{align*}
\log 2 - h(p)=2\left(\frac{k}{n}\right)^2+\frac{4}{3}\left(\frac{k}{n}\right)^4+\Oo(\left(k/n\right)^6). 
\end{align*}
Substituting this into \eqref{binomial-stirling} and noting $\sqrt{1/(2\pi n p(1-p))}=\sqrt{2/(\pi n)}(1+\Oo(k^2/n^2))$ and $k^4/n^3\leq n^{-3/5}=o(1)$ for $|k|\leq n^{3/5}$ completes the proof.
\end{proof}

\begin{proof}[Proof of \Cref{thm:main1}] Write $a\coloneqq\inf_n a_n>a_0$. Recall $Z_n(a)$ and $\mathcal A(\sigma)$ defined in \eqref{Zn}. By the Paley--Zygmund inequality, 
\begin{align}
\mathbb P\{\disc(f_1, \ldots, f_n)\leq a_n\} = \mathbb P\{Z_n(a_n) >0\}\geq\frac{\E[Z_n(a_n)]^2}{\E[Z^2_n(a_n)]}&= \frac{\sum_{\sigma, \sigma'}\mathbb P\{\mathcal A(\sigma)\}\mathbb P\{\mathcal A(\sigma')\}}{\sum_{\sigma, \sigma'}\mathbb P\{\mathcal A(\sigma)\cap \mathcal A(\sigma')\}}\nonumber\\
&= \frac{\sum_\sigma \mathbb P\{\mathcal A(1)\}^2}{\sum_\sigma\mathbb P\{\mathcal A(1)\cap \mathcal A(\sigma)\}}, \label{pzb}
\end{align}
where the last step follows from $\mathbb P\{\mathcal A(\sigma)\}= \mathbb P\{\mathcal A(1)\}$ and that $\sum_{\sigma'}\mathbb P\{\mathcal A(\sigma)\cap \mathcal A(\sigma')\}$ is independent of $\sigma$. Without loss of generality, we assume that $n$ is even and let $\Ss_k=\{\sigma\in\{\pm 1\}^n: \sum_{i}\sigma_i=2k\}$; for odd $n$ one defines $\Ss_k=\{\sigma\in\{\pm 1\}^n: \sum_{i}\sigma_i=2k-1\}$ and the discussion is identical. In what follows, we first present a Laplace-principle heuristic argument for the threshold $a_0$ and then provide a rigorous proof of the full statement. 

\paragraph{Heuristics on $a_0$.}
It is more convenient to work with the reciprocal of the lower bound in \eqref{pzb} and bound it from above: 
\begin{equation}\label{recip}
\begin{aligned}
\frac{\sum_\sigma\mathbb P\{\mathcal A(1)\cap \mathcal A(\sigma)\}}{\sum_\sigma \mathbb P\{\mathcal A(1)\}^2} &= \frac{\sum_{|k|\leq n/2}\sum_{\sigma\in\Ss_k}\mathbb P\{\mathcal A(1)\cap \mathcal A(\sigma)\}}{2^n \mathbb P\{\mathcal A(1)\}^2}\\
&= \sum_{|k|\leq n/2}\frac{1}{2^n}{\binom{n}{\frac{n}{2}+k}}\frac{\mathbb P\{\mathcal A(1)\cap \mathcal A(\sigma^k)\}}{\mathbb P\{\mathcal A(1)\}^2},
\end{aligned}
\end{equation}
where $\sigma^k\in\{\pm 1\}^n$ denotes the element with the first $n/2+k$ coordinates being $1$ and the remaining coordinates being $-1$. 
The probability of the event $\mathcal A(1)\cap \mathcal A(\sigma^k)$ concerns the small ball probability of the extreme value of a correlated planar Brownian motion with correlation $2k/n$, which can be analyzed using \Cref{lm:corB}. Therefore, the summands in \eqref{recip} are governed by a competition between the entropy cost of an overlap and the gain in correlated small-ball probability. Writing
\[
\rho=\frac{2k}{n},
\qquad
I(\rho):=\log 2-h\!\left(\frac{1+\rho}{2}\right),
\]
where $h$ is the binary entropy function in \eqref{bi-entropy}, the exponential contribution associated with overlap $\rho$ can be formally computed by plugging the asymptotics in \eqref{binomial-stirling} (this is formal because the asymptotics hold only for $|\rho|$ away from $1$): 
\[
\exp\left\{
n\left[
\frac{\lambda_1(0)-\lambda_1(\rho)}{2a_n^2}
-I(\rho)
\right]
\right\}.
\]
To ensure \eqref{recip} has order $\Oo(1)$, the exponential rate associated with every fixed overlap $\rho\in(0,1]$ must be nonpositive, suggesting the following variational threshold: 
\[
a_c^2
:=
\sup_{0<\rho\le 1}
\frac{\lambda_1(0)-\lambda_1(\rho)}
     {2I(\rho)}.
\]
The function $\lambda_1(\rho)$ does not admit an explicit formula but can be approximated by the lower bound \eqref{newguy}. Assuming this approximation is sufficiently accurate (it is exact at the two endpoints $\rho=0,1$), the supremum is attained at the endpoint $\rho=1$, yielding $a_c=a_0$. This heuristic computation turns out to identify the correct threshold, as confirmed by the rigorous analysis.

\paragraph{Full proof.}
Fix $a>a_0$ and let $\delta>0$ be some constant chosen later. To upper bound \eqref{recip}, we write it as a sum in three different regimes and control each, respectively:
\begin{align*}
\sum_{|k|\leq n/2}\frac{1}{2^n}{\binom{n}{\frac{n}{2}+k} }\frac{\mathbb P\{\mathcal A(1)\cap \mathcal A(\sigma^k)\}}{\mathbb P\{\mathcal A(1)\}^2} = \underbrace{\sum_{(1/2-\delta)n\leq |k|\leq n/2}\bullet}_{\mathrm{(i)}} \ + \underbrace{\sum_{n^{3/5}\leq |k|\leq (1/2-\delta)n}\bullet}_{\mathrm{(ii)}} \ + \underbrace{\sum_{0\leq |k|\leq n^{3/5}}\bullet}_{\mathrm{(iii)}}. 
\end{align*}
The following upper bound on the binomial coefficients will be frequently used: For every $m\leq n$, 
\begin{align}
\frac{1}{2^n}{n\choose m}\leq\exp\left\{-n(\log 2- h(m/n))\right\}, \label{exact-bi}
\end{align}
where $h$ is the binary entropy function in \eqref{bi-entropy}. 

\noindent{\em{\underline{Bounding Term (i)}.}} A direct computation using \eqref{exact-bi} gives
\begin{equation}\label{bdd-1}
\begin{aligned}
\sum_{(1/2-\delta)n\leq |k|\leq n/2}\frac{1}{2^n}{\binom{n}{\frac{n}{2}+k}}\frac{\mathbb P\{\mathcal A(1)\cap \mathcal A(\sigma^k)\}}{\mathbb P\{\mathcal A(1)\}^2}&\leq\sum_{(1/2-\delta)n\leq |k|\leq n/2}\frac{1}{2^n}{\binom{n}{\frac{n}{2}+k}}\frac{1}{\mathbb P\{\mathcal A(1)\}}\\
&\leq\exp\left\{-n\left(\log 2- h(1-\delta)-\frac{\pi^2}{8a_n^2}\right)+\log(2\delta n)\right\},   
\end{aligned}
\end{equation}
where for the last step we have used the exact asymptotics for $\mathbb P\{\mathcal A(1)\}$ in \eqref{exact}. 
The upper bound in \eqref{bdd-1} is of order $o(1)$ if $r(\delta)\coloneqq\log 2- h(1-\delta)-\frac{\pi^2}{8a_n^2}>0$. Since
\[
r(0)=\log2-\frac{\pi^2}{8a_n^2}
\geq \log2-\frac{\pi^2}{8a^2}>0
\]
uniformly in $n$, continuity of $h$ at $1$ allows us to choose a single $\delta>0$, depending only on $a$, such that $r(\delta)>0$ for every $n$. Hence (i) $=o(1)$. We fix this $\delta$ below. 

\medskip

\noindent{\em{\underline{Bounding Term (ii)}.}} For $k\leq (1/2-\delta)n$, the uniform asymptotics in \Cref{lm:corB} can be applied to obtain
\begin{align*}
\mathbb P\{\mathcal A(1)\cap \mathcal A(\sigma^k)\} &= C_1(2k/n) \exp\left\{-\frac{\lambda_1(2k/n)n}{2a_n^2}\right\} (1 + o(1)),\\ 
\mathbb P\{\mathcal A(1)\}^2&=C_1(0) \exp\left\{-\frac{\lambda_1(0)n}{2a_n^2}\right\} (1 + o(1)),
\end{align*}
where $o(1)$ depends on $\delta$. 
Moreover, by the positivity and continuity of $C_1(\rho)$, $C^* \coloneqq \max_{|\rho|\leq 1-2\delta}C_1(\rho)/C_1(0)<\infty$. Combining the above estimates with \eqref{newguy} and \eqref{exact-bi}, 
\begin{equation}\label{bdd-2}
\begin{aligned}
&\sum_{n^{3/5}\leq |k|\leq (1/2-\delta)n}\frac{1}{2^n}{\binom{n}{\frac{n}{2}+k}}\frac{\mathbb P\{\mathcal A(1)\cap \mathcal A(\sigma^k)\}}{\mathbb P\{\mathcal A(1)\}^2}\\
\leq&\ \sum_{n^{3/5}\leq |k|\leq (1/2-\delta)n}2C^*\exp\left\{-n\left[\frac{1}{2a_n^2}(\lambda_1(2k/n)-\lambda_1(0))+\log 2 - h(1/2+k/n)\right]\right\}\\
\leq&\ \sum_{n^{3/5}\leq |k|\leq (1/2-\delta)n}2C^*\exp\left\{-n\left[\log 2-\frac{\pi^2}{8a_n^2}\left(1-\sqrt{1-\left(\frac{2k}{n}\right)^2}\right) - h(1/2+k/n)\right]\right\}\\
\leq&\ \sum_{n^{3/5}\leq |k|\leq (1/2-\delta)n}2C^*\exp\left\{-n\left(\log 2\cdot\sqrt{1-\left(\frac{2k}{n}\right)^2}- h(1/2+k/n)\right)\right\}\\
\leq&\ \exp\left\{-n\min_{\rho\in [2n^{-2/5}, 1-2\delta]}s(\rho)+\log (C^*n)\right\},\quad\quad s(\rho)\coloneqq \log 2 \cdot \sqrt{1-\rho^2}- h(1/2+\rho/2), 
\end{aligned}
\end{equation}
where the third step is obtained by replacing $a_n$ by the lower bound $a_0 = \pi/\sqrt{8\log 2}$ and the last step uses the symmetry of $s(\rho)$. For $0<\rho<1$,
\[
s'(\rho)
=
\operatorname{arctanh}(\rho)
-
(\log 2)\frac{\rho}{\sqrt{1-\rho^2}}.
\]
Set
\[
F(\rho)
=
\frac{\sqrt{1-\rho^2}}{\rho}
\operatorname{arctanh}(\rho).
\]
A direct differentiation gives
\[
F'(\rho)
=
\frac{\rho-\operatorname{arctanh}(\rho)}
{\rho^2\sqrt{1-\rho^2}}
<0.
\]
Furthermore, $F(0+)=1$ and $F(1-)=0$. Since $\log2\in(0,1)$,
there is a unique $\rho_\star\in(0,1)$ such that
$F(\rho_\star)=\log2$. Hence $s'>0$ on $(0,\rho_\star)$ and
$s'<0$ on $(\rho_\star,1)$. Since $s(0)=s(1)=0$, for any $[\rho_1, \rho_2]\subseteq [0, 1]$, $s(\rho)$ attains its minimum on $[\rho_1, \rho_2]$ at either $\rho_1$ or $\rho_2$. 
Consequently, 
\begin{align*}
\min_{\rho\in [2n^{-2/5}, 1-2\delta]}s(\rho) = \min\{s(2n^{-2/5}), s(1-2\delta)\}. 
\end{align*}
Note that $s(1-2\delta)>s(1)=0$ which is independent of $n$ and $s(2n^{-2/5})\to s(0) = 0$ as $n\to\infty$. For all sufficiently large $n$, the minimum is governed by $s(2n^{-2/5})$. Taylor expanding $s(\rho)$ at $\rho = 0$ and observing $s(0)=s'(0) = 0$ and $s''(0) = 1-\log 2$ yields that  
\begin{align*}
s(2n^{-2/5}) = 2(1-\log 2) n^{-4/5} + o(n^{-4/5}). 
\end{align*}
Substituting this into the bound in \eqref{bdd-2} concludes that (ii) = $o(1)$. 

\medskip

\noindent{\em{\underline{Bounding Term (iii)}.}}
Controlling Term (iii) requires the exact asymptotics in \Cref{lem:binom} for the binomial coefficients near the center using the Stirling formula. For $|k|\le n^{3/5}$, $2k/n\leq 2n^{-2/5}\to 0$. Since $\lambda_1$ is even and analytic in a neighborhood of 0, there is a constant $C<\infty$ such that, uniformly for $|k|\le n^{3/5}$,
\[
\lambda_1(2k/n)-\lambda_1(0) =
\frac{\lambda_1''(0)}{2}\left(\frac{2k}{n}\right)^2 + \mathcal O((2k/n)^4).
\]
Moreover, continuity of $C_1$ at $0$ gives
\[
\sup_{|k|\le n^{3/5}}\left|
\frac{C_1(2k/n)}{C_1(0)}-1
\right|=o(1).
\]
Then the regularity properties in \Cref{lm:corB}~(i) give the following exact small-ball asymptotics uniformly on every fixed compact subset of $(-1,1)$. Since $2k/n\le 2n^{-2/5}$ and $a_n/\sqrt n\to 0$, there exists a deterministic sequence \(\varepsilon_n\to0\) such that, uniformly for
\(|k|\le n^{3/5}\),
\[
\mathbb P\{\mathcal A(1)\cap \mathcal A(\sigma^k)\} =
C_1(2k/n)\exp\left\{-\frac{\lambda_1(2k/n)n}{2a_n^2}
\right\} (1+\mathcal O(\varepsilon_n)),
\]
and
\[
\mathbb P\{\mathcal A(1)\}^2 = C_1(0) \exp\left\{ -\frac{\lambda_1(0)n}{2a_n^2} \right\} (1+\mathcal O(\varepsilon_n)).
\]

Consequently, combining these estimates with \Cref{lem:binom} yields that, uniformly for $|k|\le n^{3/5}$,
\[
\frac{1}{2^n}
\binom{n}{\frac{n}{2}+k}
\frac{
\mathbb P\{\mathcal A(1)\cap \mathcal A(\sigma^k)\}
}{\mathbb P\{\mathcal A(1)\}^2} =
(1+o(1))
\sqrt{\frac{2}{\pi n}}
\exp\left\{
-\frac{2k^2}{n}
\left(
1-\frac{\beta}{a_n^2}
\right)
+
\mathcal O\left(
\frac1n+\frac{k^4}{n^3}
\right)
\right\},
\]
where the $o(1)$ is uniform over $|k|\le n^{3/5}$. Since $\sup_{|k|\le n^{3/5}}\left(\frac1n+\frac{k^4}{n^3}\right)=o(1)$, the final exponential error can be absorbed into this uniform $1+o(1)$ factor.
Consequently, denoting 
\begin{align}
\beta \coloneqq -\frac{\lambda''_1(0)}{2}>0\label{beta}
\end{align}
and noting $\inf_{n}a_n>a_0>\sqrt{\beta}$, 
\begin{equation}\label{bdd-3}
\begin{aligned}
\sum_{|k|\leq n^{3/5}}\frac{1}{2^n}{n\choose \frac{n}{2}+k}\frac{\mathbb P\{\mathcal A(1)\cap \mathcal A(\sigma^k)\}}{\mathbb P\{\mathcal A(1)\}^2}&=\sum_{|k|\leq n^{3/5}}(1+o(1))\sqrt{\frac{2}{\pi n}}\exp\left\{-\frac{2k^2}{n}\left(1-\frac{\beta}{a_n^2}\right)+\mathcal O\left(\frac{1}{n}+\frac{k^4}{n^3}\right)\right\}\\
&=\sum_{|k|\leq n^{3/5}}(1+o(1))\sqrt{\frac{2}{\pi n}}\exp\left\{-\frac{2k^2}{n}\left(1-\frac{\beta}{a_n^2}\right)\right\}\\
&=\frac{1+o(1)}{\sqrt{1-\beta/a_n^2}},  
\end{aligned}
\end{equation}
where the last step follows from the Riemann-sum approximation, with an error that is $o(1)$ uniformly in $a_n$ under the assumption $\inf_n a_n>a_0>\sqrt{\beta}$.  Combining the estimates in \eqref{bdd-1}, \eqref{bdd-2}, and \eqref{bdd-3} yields that, 
\begin{align}
\frac{\E[Z^2_n(a_n)]}{\E[Z_n(a_n)]^2} &= \frac{1+o(1)}{\sqrt{1-\beta/a_n^2}}\implies \mathbb P\{\disc(f_1, \ldots, f_n)\leq a_n\}\geq \sqrt{1-\beta/a_n^2} - o(1), \label{hr}
\end{align}
where $o(1)$ depends on $a$ and $n$. 
The proof is complete. 
\end{proof}

\subsection{Proof of \Cref{lm:corB}}\label{sec:corB}
\paragraph{Proof of statement (i).} For fixed $\rho$, the asymptotics in \eqref{goal} can be obtained using classical heat kernel estimates \citep{davies1989heat}. Since we need an asymptotic formula uniform in $\rho$, we need to track the errors in various approximations leading to \eqref{goal}. 

Denote by $W(t)$ the standard planar Brownian motion started from the origin and $T = \e^{-2}$. Since $B(t)$ has the same law as $T(\rho)W(t)$, 
\begin{align*}
\Pp\left\{B(t)\in [-\e, \e]^2, \forall t\in [0, 1]\right\}&= \mathbb{P}\left\{W(t)\in \e\mathcal R(\rho), \forall t\in [0, 1]\right\}=\mathbb{P}\left\{W(t)\in \mathcal R(\rho), \forall t\in [0, T]\right\},
\end{align*}
where the last step follows from Brownian scaling: $\varepsilon^{-1}W(\varepsilon^2 t)$ is again a standard planar Brownian motion.
The last term is the survival probability of a standard planar Brownian motion killed on $\partial\mathcal R(\rho)$ up to time $T$. If we denote the transition kernel of this Brownian motion by $p_t(x, y)$, then  
\begin{align*}
\mathbb P\left\{W(t)\in \mathcal R(\rho), \forall t\in [0, T]\right\} = \int_{\mathcal R(\rho)}p_T(x)\mathrm d x, \quad\quad p_t(x)\coloneqq p_t(0, x).  
\end{align*}

To understand this quantity, we analyze the spectrum of the kernel $p_t(x, y)$. For $|\rho|\leq 1-\delta$, $\mathcal R(\rho)$ remains uniformly bounded, i.e., 
\begin{equation}\label{allhere}
\begin{aligned}
\max_{|\rho|\leq 1-\delta}|\mathcal R(\rho)|
&= \max_{|\rho|\leq 1-\delta}\frac{4}{\sqrt{1-\rho^2}}
= \frac{4}{\sqrt{\delta(2-\delta)}}
\leq \frac{4}{\sqrt{\delta}},\\
\max_{|\rho|\leq 1-\delta}\mathrm{diam}(\mathcal R(\rho))
&= \max_{|\rho|\leq 1-\delta}2\sqrt{\frac{2}{1-|\rho|}}
= \sqrt{\frac{8}{\delta}}.
\end{aligned}
\end{equation}
Let $0<\lambda_1(\rho) < \lambda_2(\rho)\dots$ be the Dirichlet eigenvalues of the negative Laplacian $(-\Delta)$ on $\mathcal R(\rho)$, and let $\{\phi_k(\cdot; \rho)\}_{k=1}^\infty$ be the corresponding $L^2$-normalized eigenfunctions.
For $|\rho|\leq 1-\delta$, the domain $\mathcal R(\rho)$ is a bounded convex polygon with uniformly controlled angles (where the uniformity depends only on $\delta$). Standard elliptic regularity gives $\phi_i\in H^2(\mathcal R(\rho)^\circ)$, hence a continuous representative on $\mathcal R(\rho)$; each $\phi_i$ is smooth in the interior.
The heat kernel $p_t(x, y)$ admits the following spectral expansion \citep{davies1989heat}: 
\begin{align}
p_t(x, y) = \sum_{i=1}^\infty e^{-\lambda_i(\rho) t/2}\phi_i(x; \rho)\phi_i(y; \rho), \label{0011}
\end{align}
where the convergence is in $L^2(\mathcal R(\rho)\times \mathcal R(\rho))$ for every $t>0$. To further upgrade convergence to the pointwise level (uniformly in both spatial variables and correlations), we need quantitative spectral estimates. First, we lower bound $\lambda_i(\rho)$ uniformly for all $|\rho|\leq 1-\delta$ \citep{li1983schrodinger}: 
\begin{align}
\lambda_i(\rho)\geq\frac{2\pi i}{|\mathcal R(\rho)|}\stackrel{\eqref{allhere}}{>}\sqrt{\delta}i, \quad\quad i=1, 2, \ldots \label{bdd1}
\end{align}
Meanwhile, the eigenfunctions have at most algebraic growth in the corresponding eigenvalues. Since $\phi_i(\cdot; \rho)$ is an eigenfunction of the kernel $p_t(x, y)$ with unit $L^2$ norm associated with eigenvalue $e^{-\lambda_i(\rho)t/2}$, by definition, for all $t>0$, 
\begin{equation}\label{2toinfty}
\begin{aligned}
|\phi_i(x; \rho)| &= e^{\lambda_i(\rho)t/2}\left|\int_{\mathcal R(\rho)}\phi_i(y; \rho)p_t(x, y) \mathrm d y\right|\\
&\leq e^{\lambda_i(\rho)t/2}\|p_t(x, \cdot)\|_{L^2(\mathcal R(\rho))}\leq e^{\lambda_i(\rho)t/2}\sqrt{p_{2t}(x, x)}\leq\frac{e^{\lambda_i(\rho)t/2}}{\sqrt{4\pi t}}, 
\end{aligned}
\end{equation}
where the penultimate step used the semigroup property and the last step follows by dropping the killing effect. 
Choosing $t = 1/\lambda_i(\rho)$ followed by taking supremum over $x$ yields 
\begin{align}
\|\phi_i(\cdot;\rho)\|_{L^\infty(\mathcal R(\rho))} \leq
\frac{e^{1/2}}{\sqrt{4\pi}}\sqrt{\lambda_i(\rho)} \leq
\sqrt{\lambda_i(\rho)}.\label{bdd2}
\end{align} 
Combining \eqref{bdd1} and \eqref{bdd2}, for all $t>1$, $x,y\in\mathcal R(\rho)$, and $|\rho|\leq1-\delta$, 
\begin{align*}
e^{-\lambda_i(\rho)t/2}
|\phi_i(x;\rho)\phi_i(y;\rho)|\leq
\lambda_i(\rho)e^{-\lambda_i(\rho)t/2}\leq
2 e^{-\lambda_i(\rho)t/4}\leq
2 e^{-\sqrt{\delta}\,it/4},
\end{align*}
where the last expression is summable in $i$. Hence, by the Weierstrass $M$-test, the series in \eqref{0011} converges uniformly. As a result, we can set $x=0$ and integrate $y$ over $\mathcal R(\rho)$ in \eqref{0011} to obtain the following series representation of the small ball probability: 
\begin{align}
\int_{\mathcal R(\rho)}p_t(x)\mathrm d x = C_1(\rho)e^{-\lambda_1(\rho)t/2} + \sum_{i=2}^\infty e^{-\lambda_i(\rho) t/2}\phi_i(0; \rho)\int_{\mathcal R(\rho)}\phi_i(x; \rho) \mathrm d x, \label{heat-series}
\end{align}
where $C_1(\rho)$ is defined in \eqref{myc1} and the convergence is uniform for $|\rho|\leq 1-\delta$. 

We next obtain a uniform remainder estimate. Under the linear change of variables from $\mathcal R(\rho)$ to the fixed square $\mathcal R(0)$, the operators form a uniformly elliptic family for $|\rho|\leq1-\delta$. The min--max principle, applied to fixed finite-dimensional subspaces of $H_0^1(\mathcal R(0))$, therefore gives constants $M_\delta<\infty$ and $M_{\delta,r}<\infty$ such that
\[
\sup_{|\rho|\leq1-\delta}\lambda_1(\rho)\leq M_\delta,
\qquad
\sup_{|\rho|\leq1-\delta}\max_{1\leq i\leq r}\lambda_i(\rho)\leq M_{\delta,r}
\]
for every fixed $r$. Set $c_\delta\coloneqq\sqrt\delta$ and choose an integer $r_\delta$ so large that $c_\delta r_\delta\geq2M_\delta$. Factoring out the principal exponential in \eqref{heat-series} and using \eqref{bdd2},
\begin{align*}
& e^{\lambda_1(\rho)t/2}
\left|\sum_{i=2}^\infty e^{-\lambda_i(\rho)t/2}\phi_i(0;\rho)
\int_{\mathcal R(\rho)}\phi_i(x;\rho)\,\mathrm dx\right|\\
&\qquad\leq 2\delta^{-1/4}
\sum_{i=2}^\infty\sqrt{\lambda_i(\rho)}
 e^{-(\lambda_i(\rho)-\lambda_1(\rho))t/2}.
\end{align*}
For $2\leq i<r_\delta$, the fundamental-gap estimate of \citet[Corollary 1.4]{andrews2011proof}
\begin{align}
g_\delta\coloneqq\inf_{|\rho|\leq1-\delta}
(\lambda_2(\rho)-\lambda_1(\rho))
\geq\inf_{|\rho|\leq1-\delta}
\frac{3\pi^2}{\operatorname{diam}(\mathcal R(\rho))^2}
\geq\frac{3\pi^2}{8}\delta>0
\label{fg}
\end{align}
and the bound $\lambda_i(\rho)\leq M_{\delta,r_\delta}$ show that the corresponding finite sum is at most
\[
C_{\delta,0}e^{-g_\delta t/2},
\qquad
C_{\delta,0}\coloneqq(r_\delta-2)\sqrt{M_{\delta,r_\delta}}.
\] 

For $i\geq r_\delta$, \eqref{bdd1} and the definition of $r_\delta$ imply
$\lambda_i(\rho)-\lambda_1(\rho)\geq\lambda_i(\rho)/2$. For all sufficiently large $t=t(\delta)$, the function $x\mapsto\sqrt{x}e^{-xt/4}$ is decreasing on $[c_\delta r_\delta,\infty)$, and hence
\begin{align*}
\sum_{i\geq r_\delta}\sqrt{\lambda_i(\rho)}
 e^{-(\lambda_i(\rho)-\lambda_1(\rho))t/2}
\leq\sum_{i\geq r_\delta}\sqrt{c_\delta i}
 e^{-c_\delta i t/4}\leq C_{\delta, 1}\delta e^{-c_\delta r_\delta t/8}
\end{align*}
for some $C_{\delta, 1}>0$. Consequently, for some $C_\delta, \gamma_\delta>0$, the remainder in \eqref{heat-series} is bounded uniformly by
\[
C_\delta e^{-\lambda_1(\rho)t/2}e^{-\gamma_\delta t}.
\]
Finally, choosing the principal eigenfunction positive, $C_1(\rho)>0$ follows by the strong maximum principle \cite[Section 6.4.2]{evans2010partial}. Furthermore, statement~(iii) (whose proof is independent of statement~(i)) implies
$\inf_{|\rho|\leq1-\delta}C_1(\rho)>0$. Dividing the preceding remainder bound by the principal term proves \eqref{goal} uniformly for all large $t$ (i.e., $t = T\to\infty$ as $\e\to 0$).

\paragraph{Proof of statement (ii).} We first establish the regularity estimates and then prove the lower bound \eqref{newguy}. Fix the sign of $\phi_1(\cdot; \rho)>0$. Since the domain depends smoothly on $\rho$, we take the transform $T(\rho): \mathcal R(\rho)\to\mathcal R(0): x \mapsto T(\rho)x$ to remove the domain's dependence on $\rho$, which is similar to \cite[Lemma 3.1]{el2007domain}. The transformed negative Laplacian is given by
\begin{align}
&\mathcal L(\rho)\coloneqq -\nabla\cdot(\Sigma(\rho)\nabla), &\Sigma(\rho)\coloneqq T(\rho)T(\rho)^\top=\begin{pmatrix}1&\rho\\ \rho&1\end{pmatrix}.
\end{align}
Its normalized principal eigenfunction is
\begin{align}
\psi_1(y;\rho)\coloneqq |\det T(\rho)|^{-1/2}\phi_1(T(\rho)^{-1}y;\rho),\qquad y\in\mathcal R(0).\label{psi-phi}
\end{align}
Thus $\mathcal L(\rho)\psi_1(\cdot; \rho)=\lambda_1(\rho)\psi_1(\cdot; \rho)$ and $\|\psi_1(\cdot;\rho)\|_{L^2(\mathcal R(0))}=1$. Since $\lambda_1(\rho)$ is simple (e.g., by \eqref{fg}) and has the variational formula:
\begin{align}
\lambda_1(\rho) = \inf_{u\in H^1_0(\mathcal R(0)): \|u\|_{L^2(\mathcal R(0))}=1}E(u; \rho),\label{var-form}
\end{align}
where 
\begin{align}\label{dirichlet}
E(u; \rho)\coloneqq\langle u, \mathcal L(\rho)u\rangle_{L^2(\mathcal R(0))}=\int_{\mathcal R(0)}(\partial_x u)^2+(\partial_y u)^2 + 2\rho\partial_x u\partial_y u\ \mathrm{d} x\mathrm{d} y. 
\end{align}
Classical perturbation results establish the real-analyticity of $\lambda_1(\rho)$ and of $\rho\mapsto\psi_1(\cdot;\rho)$ in $L^2(\mathcal R(0))$ \cite[Chapter 7]{kato1966perturbation}. Since $E(u(x, y); \rho) = E(u(x, -y); -\rho)$, $\lambda_1(\rho)$ is even by the formula \eqref{var-form} and hence $\lambda_1'(0)=0$. On the other hand, since $E(u; \rho)$ is linear in $\rho$ for every fixed $u$, $\lambda_1(\rho)$ is concave on $(-1, 1)$. Consequently, $\lambda_1'(\rho)\leq0$ for $\rho\geq0$ and $\lambda_1''(\rho)\leq0$ wherever the second derivative is evaluated. The inequality for $\lambda_1'(\rho)$ is strict for every $\rho>0$: analyticity together with $\lambda_1''(0)<0$ gives $\lambda_1'(\rho)<0$ for all sufficiently small $\rho>0$, and concavity makes $\lambda_1'$ nonincreasing on the remainder of $(0,1)$. The strict inequality $\lambda_1''(0)<0$ follows from the second-order perturbation formula for the simple principal eigenvalue: 
\begin{align*}
\lambda_1''(0)
&=\langle\psi_1(\cdot;0),\mathcal L''(0)\psi_1(\cdot;0)\rangle_{L^2(\mathcal R(0))}+2\sum_{j\geq2}
\frac{\langle\psi_j(\cdot;0),\mathcal L'(0)\psi_1(\cdot;0)\rangle^2_{L^2(\mathcal R(0))}}{\lambda_1(0)-\lambda_j(0)}\\
& = 2\sum_{j\geq2}\frac{\langle\psi_j(\cdot;0),\mathcal L'(0)\psi_1(\cdot;0)\rangle^2_{L^2(\mathcal R(0))}}{\lambda_1(0)-\lambda_j(0)}<0&\text{(since $\mathcal L''(0)=0$)}. 
\end{align*}

Moreover, since $\mathcal R(0)$ is a square in $\R^2$, the Dirichlet eigenvalues and eigenfunctions are explicitly known; see \cite[Section 3.1]{grebenkov2013geometrical}. The eigenvalues can be written in a double-index form as $\pi^2(i^2+j^2)/4$ with associated eigenfunctions $p_i(x)p_j(y)$, where $p_i(x)=\cos(i\pi x/2)$ if $i$ is odd and $p_i(x)=\sin(i\pi x/2)$ if $i$ is even. In this basis, the perturbation series admits an explicit form. Since $\mathcal L'(0)=-2\partial_x\partial_y$ and $\psi_1(x,y;0)=p_1(x)p_1(y)$,
\begin{align*}
\mathcal L'(0)\psi_1(\cdot;0) = -\frac{\pi^2}{2}\sin(\pi x/2)\sin(\pi y/2).
\end{align*}
By parity, $\langle p_i, \sin(\pi x/2)\rangle_{L^2(-1,1)}=0$ for odd $i$, and a direct computation gives, for $i=2k$,
\begin{align*}
c_{2k} \coloneqq \int_{-1}^1 \sin(k\pi x)\sin\Big(\frac{\pi x}{2}\Big)\mathrm dx
= (-1)^{k-1}\frac{8k}{\pi(4k^2-1)}.
\end{align*}
Hence only the modes $(i, j)=(2k, 2\ell)$ with $k, \ell\geq1$ couple to $\psi_1(\cdot; 0)$, with eigenvalues $\pi^2(k^2+\ell^2)$, and the perturbation formula becomes
\begin{align}\label{eq:lambda-pp-series}
-\lambda_1''(0) = \frac{\pi^2}{2}\sum_{k, \ell\geq1}\frac{c_{2k}^2c_{2\ell}^2}{k^2+\ell^2-1/2} = \frac{2048}{\pi^2}\sum_{k, \ell\geq1}\frac{k^2\ell^2}{(4k^2-1)^2(4\ell^2-1)^2(k^2+\ell^2-1/2)}.
\end{align}
The series converges absolutely, i.e., its terms are dominated by $\frac{2}{81}k^{-2}\ell^{-2}(k^2+\ell^2)^{-1}$ because $4k^2-1\geq3 k^2$ and $k^2+\ell^2-1/2\geq (k^2+\ell^2)/2$. Truncating \eqref{eq:lambda-pp-series} at $k, \ell\leq N$, 
\begin{align*}
&\sum_{\max(k,\ell)>N}
\frac{1}{k^2\ell^2(k^2+\ell^2)} 
 \leq
2\sum_{k>N}\sum_{\ell\geq1}
\frac{1}{k^4\ell^2}
\leq
2\left(\frac{1}{3N^3}\right)
\left(\frac{\pi^2}{6}\right).
\end{align*}
Consequently,
\begin{align*}
0
&\leq
\frac{2048}{\pi^2}
\sum_{\max(k,\ell)>N}
\frac{k^2\ell^2}
{(4k^2-1)^2(4\ell^2-1)^2(k^2+\ell^2-1/2)} 
\leq\frac{2048}{\pi^2}\cdot
\frac{2}{81}\cdot
\frac{\pi^2}{9N^3}
\leq
\frac{6}{N^3}.
\end{align*}
Evaluating the finite sum at $N=50$ in exact rational arithmetic, using $333/106<\pi<355/113$, and combining it with the tail bound $6/N^3$ yields the rigorous enclosure $-\lambda_1''(0)\in[1.965,1.966]$. 

We next use the variational formula to obtain the lower bound on $\lambda_1(\rho)$ in \eqref{newguy}. 
By symmetry, it suffices to consider $\rho\in [0, 1)$. 
Let $\theta\in [0, \pi/4)$ such that $\rho = \sin(2\theta)$, and define the unit directions $b_1 = (\cos(\theta), \sin(\theta))^\top$ and $b_2 = (\sin(\theta), \cos(\theta))^\top$. 
For any $u\in H^1_0(\mathcal R(0))$ with $\|u\|_{L^2(\mathcal R(0))}=1$, its associated Dirichlet form \eqref{dirichlet} can be represented as a sum of directional Dirichlet forms along $b_1$ and $b_2$:  
\begin{align}
E(u; \rho)=\int_{\mathcal R(0)}(b_1\cdot\nabla u)^2 + (b_2\cdot\nabla u)^2\ \mathrm{d} x\mathrm{d} y. \label{b1b2}
\end{align}
By the Fubini--Tonelli theorem, we can evaluate each integral under a rotation of the basis. In particular, for the first integral, we evaluate it under the basis $(b_1, b_1^\perp)^\top$, where $b_1^\perp$ is a unit vector perpendicular to $b_1$. For each fixed coordinate $t_2$ in $b_1^\perp$, the integral along $b_1$ is $\partial_{t_1}u(t_1b_1+t_2b_1^\perp)$, and the support of $t_1$ is an interval bounded by $2/\cos(\theta)$. An application of the 1D-Poincaré inequality yields
\begin{align}
\int_{\mathcal R(0)}(b_1\cdot\nabla u)^2\ \mathrm{d} x\mathrm{d} y &= \int_{t_1b_1+t_2b_1^\perp\in\mathcal R(0)}(\partial_{t_1} u(t_1b_1+t_2b_1^\perp))^2\ \mathrm{d} t_1\mathrm{d} t_2\nonumber\\
&\geq\frac{\pi^2\cos(\theta)^2}{4}\int_{t_1b_1+t_2b_1^\perp\in\mathcal R(0)}u(t_1b_1+t_2b_1^\perp)^2\ \mathrm{d} t_1\mathrm{d} t_2\nonumber\\
&= \frac{\pi^2\cos(\theta)^2}{4},  \label{lbb1}
\end{align}
where the last identity uses that $\|u\|_{L^2(\mathcal R(0))}=1$. Similarly, 
\begin{align}
\int_{\mathcal R(0)}(b_2\cdot\nabla u)^2\ \mathrm{d} x\mathrm{d} y \geq\frac{\pi^2\cos(\theta)^2}{4}.  \label{lbb2}
\end{align}
Substituting \eqref{lbb1}, \eqref{lbb2} into \eqref{b1b2} and combining with \eqref{var-form} yields the desired lower bound on $\lambda_1(\rho)$.

\paragraph{Proof of statement (iii).} Establishing the continuity of $C_1(\rho)$ requires slightly more work due to the pointwise evaluation. Using the definition of $C_1(\rho)$ and \eqref{psi-phi}, we can rewrite $C_1(\rho)$ as 
\begin{align}
C_1(\rho)=\psi_1(0;\rho)\int_{\mathcal R(0)}\psi_1(y;\rho)\,\mathrm dy.
\end{align}
To establish the continuity of $C_1(\rho)$, it remains to show that $\psi_1(0;\rho)$ is continuous in $(-1, 1)$. 

Fix $\delta>0$. Let $\Omega_1\Subset\Omega_2\Subset\mathcal R(0)$ be open balls centered at the origin with radii $0.2$ and $0.5$, respectively. For any $\rho, \rho'\in [-1+\delta, 1-\delta]$ and $f\coloneqq f(\cdot; \rho, \rho') \coloneqq \psi_1(\cdot; \rho)-\psi_1(\cdot; \rho')$,
\begin{align*}
\mathcal L(\rho)f(\cdot; \rho, \rho') &= \mathcal L(\rho)\psi_1(\cdot; \rho)-\mathcal L(\rho')\psi_1(\cdot; \rho') + (\mathcal L(\rho')-\mathcal L(\rho))\psi_1(\cdot; \rho')\\
& = \underbrace{\lambda_1(\rho)\psi_1(\cdot; \rho)-\lambda_1(\rho')\psi_1(\cdot; \rho')}_{\eqqcolon f_1}+ \underbrace{(\mathcal L(\rho')-\mathcal L(\rho))\psi_1(\cdot; \rho')}_{\eqqcolon f_2}.
\end{align*}
For $|\rho|\leq 1-\delta$,  $\mathcal L(\rho)$ is uniformly elliptic. By the interior $H^2$ regularity estimate on $\Omega_2\Subset\mathcal R(0)$ \cite[Section 6.3.1, Theorem 1]{evans2010partial} and that $\psi_1(\cdot; \rho)$ is $L^2(\mathcal R(0))$-normalized, 
\begin{equation}\label{psi_1_h2}
\begin{aligned}
\sup_{|\eta|\leq 1-\delta}\|\psi_1(\cdot;\eta)\|_{H^2(\Omega_2)}
&\lesssim
\sup_{|\eta|\leq 1-\delta}
\left(
\|\mathcal L(\eta)\psi_1(\cdot;\eta)\|_{L^2(\mathcal R(0))}
+\|\psi_1(\cdot;\eta)\|_{L^2(\mathcal R(0))}
\right)\\
&\leq
\max_{|\eta|\leq 1-\delta}\lambda_1(\eta)+1<\infty.
\end{aligned}
\end{equation}
Applying the same estimate on $\Omega_1\Subset\Omega_2$ to $f$ yields
\begin{align*}
\|f\|_{H^2(\Omega_1)}
&\lesssim \|f_1\|_{L^2(\Omega_2)}+\|f_2\|_{L^2(\Omega_2)}+\|f\|_{L^2(\Omega_2)}\\
&\lesssim
\left(\max_{|\eta|\leq 1-\delta}\lambda_1(\eta)+1\right)\|f\|_{L^2(\mathcal R(0))}
+|\lambda_1(\rho)-\lambda_1(\rho')|+|\rho-\rho'|.
\end{align*}
Here, we used
\[
\|f_1\|_{L^2(\Omega_2)}
\leq
\left(\max_{|\eta|\leq 1-\delta}\lambda_1(\eta)\right)\|f\|_{L^2(\Omega_2)}
+|\lambda_1(\rho)-\lambda_1(\rho')|,
\]
and
\begin{align*}
\|f_2\|_{L^2(\Omega_2)}
&=
\left\|\operatorname{tr}\!\left((\Sigma(\rho)-\Sigma(\rho'))\nabla^2\psi_1(\cdot;\rho')\right)\right\|_{L^2(\Omega_2)}\\
&\leq
\max_{|\eta|\leq 1-\delta}\left\|\frac{\mathrm d\Sigma}{\mathrm d\eta}\right\|_F
|\rho-\rho'|
\sup_{|\eta|\leq 1-\delta}\|\psi_1(\cdot;\eta)\|_{H^2(\Omega_2)}
\lesssim |\rho-\rho'|.
\end{align*}
and for the third inequality we applied \eqref{psi_1_h2}. This shows that $f|_{\Omega_1}\in H^2(\Omega_1)$ for every $\rho, \rho'\in [-1+\delta, 1-\delta]$. Taking $\rho'\to\rho$ and recalling that $\rho\mapsto\psi_1(\cdot; \rho)$ is analytic in $L^2(\mathcal R(0))$ shows that $\rho\mapsto\psi_1(\cdot; \rho)|_{\Omega_1}$ is continuous in $H^2(\Omega_1)$. To finish the proof, we pass this continuity to the pointwise estimate via the Sobolev inequalities \cite[Section 5.6.3, Theorem 6]{evans2010partial}: 
\begin{align*}
|\psi_1(0, \rho')-\psi_1(0, \rho)|\leq\|f(\cdot; \rho, \rho')\|_{C(\overline{\Omega}_1)}
\lesssim\|f(\cdot; \rho, \rho')\|_{H^2(\Omega_1)} \xrightarrow{\rho'\to\rho} 0.
\end{align*}

%% file: solution-space.tex
\section{Geometry of solution space at $H=1/2$}\label{sec:ogp}

\subsection{Proof of \Cref{thm:count}}
Let $\bar{Y}_n(a_n)\coloneqq Z_n(a_n)-Y_n(a_n)$ count the number of elements counted by $Z_n(a_n)$ (see \eqref{Zn} for a definition) that are not local minima:   
\begin{align}
&\bar{Y}_n(a_n)\coloneqq\sum_{\sigma\in\{\pm 1\}^n}\mathbf 1\{\mathcal E(\sigma)\}, &\mathcal E(\sigma)\coloneqq \left\{\text{$L(\sigma^{(i)})<L(\sigma)\leq a_n$ for some $i$}\right\}.
\end{align}
For $a_0<a_n\ll n^{1/4}/\sqrt{\log n}$, the result in \eqref{expec} would follow if we show $\E[\bar{Y}_n(a_n)]/\E[Z_n(a_n)] = o(1)$. 
To this end, we apply a first-moment computation. By a union bound, we have 
\begin{align}
\E[\bar{Y}_n(a_n)]=\sum_{\sigma\in\{\pm 1\}^n}\Pp\{\mathcal E(\sigma)\}\leq\sum_{\sigma\in\{\pm 1\}^n}\sum_{i=1}^n\Pp\{\mathcal A(\sigma)\cap \mathcal A(\sigma^{(i)})\} = 2^n\cdot n\cdot\Pp\{\mathcal A(1)\cap \mathcal A(1^{(i)})\}. \label{1stbdd}
\end{align}
As we have seen in the proof of \Cref{thm:main1}, $\Pp\{\mathcal A(\sigma)\cap \mathcal A(\sigma^{(i)})\}$ corresponds to the survival probability of a correlated planar Brownian motion started from the origin killed on the boundary of $[-a_n/\sqrt{n}, a_n/\sqrt{n}]^2$ before time $1$. However, since $1$ and $1^{(i)}$ have correlation $1-2/n$, which converges to one as $n\to\infty$, the results of \Cref{lm:corB} do not apply directly. Nevertheless, since we only need an upper bound here, we can still adapt part of the machinery developed there to obtain the following useful estimate.
\begin{lemma}
Let $B(t) = (B_1(t), B_2(t))^\top$ be a planar Brownian motion started from the origin with constant correlation $\rho\in (-1, 1)$. Let $\lambda_1(\rho)>0$ be the principal Dirichlet eigenvalue of the negative Laplacian $(-\Delta)$ on $\mathcal R(\rho)$, where $\mathcal R(\rho)$ is defined in \eqref{Rrho}. For $0<\e<1$,   
\begin{align}
\Pp\left\{B(t)\in [-\e, \e]^2, \forall t\in [0, 1]\right\} \leq \frac{1}{\sqrt{\pi}(1-\rho^2)^{1/4}}\exp\left\{-\frac{\lambda_1(\rho)}{2}\left(\frac{1}{\e^2}-1\right)\right\}. \label{uppbdd}
\end{align}
\begin{proof}
Let $T = \e^{-2}$. 
Denote by $W(t)$ the standard planar Brownian motion started from the origin and killed on $\partial\mathcal R(\rho)$ (see \eqref{Rrho} for the definition of $\mathcal R(\rho)$), and $p_t(x, y)$ the corresponding transition kernel. Let $P_t: f\mapsto \int_{\mathcal R(\rho)}f(y)p_t(x, y) \mathrm{d} y$ be the semigroup of $W_t$. Following a similar argument as in the proof of \Cref{lm:corB} and the semigroup property, 
\begin{align*}
\Pp\left\{B(t)\in [-\e, \e]^2, \forall t\in [0, 1]\right\}=(P_T1)(0) &= (P_1P_{T-1}1)(0)\\
&= \int_{\mathcal R(\rho)}p_1(0, y) (P_{T-1}1)(y) \mathrm{d} y\\
&\leq \|p_1(0, \cdot)\|_{L^2(\mathcal R(\rho))}\|P_{T-1}1\|_{L^2(\mathcal R(\rho))}\\
&\leq \|p_1(0, \cdot)\|_{L^2(\mathcal R(\rho))}\|P_{T-1}\|_{L^2(\mathcal R(\rho))\to L^2(\mathcal R(\rho))}|\mathcal R(\rho)|^{1/2}\\
&\leq \frac{1}{\sqrt{4\pi}}e^{-\lambda_1(\rho)(T-1)/2}|\mathcal R(\rho)|^{1/2}, 
\end{align*}
where the last step follows from the last inequality in \eqref{2toinfty} and $\|P_{t}\|_{L^2(\mathcal R(\rho))} = e^{-\lambda_1(\rho)t/2}$ for $t>0$. 
The proof is finished by noting $|\mathcal R(\rho)|=4/\sqrt{1-\rho^2}$. 
\end{proof}
\end{lemma}
Applying \eqref{uppbdd} to $\Pp\{\mathcal A(\sigma)\cap \mathcal A(\sigma^{(i)})\}$ and replacing $\lambda_1(\rho)$ by its lower bound in \eqref{newguy},  
\begin{align} 
\Pp\{\mathcal A(\sigma)\cap \mathcal A(\sigma^{(i)})\}\leq\frac{n^{1/4}}{\sqrt{\pi}2^{1/4}}\exp\left\{-\frac{\pi^2n}{8a_n^2}\left(1+\sqrt{1-(1-2/n)^2}\right)\left(1-\frac{a_n^2}{n}\right)\right\}. 
\end{align} 
Substituting this into \eqref{1stbdd} and applying the Taylor expansion yields
\begin{align*}
\E[\bar{Y}_n(a_n)]&\leq\exp\left\{-\frac{\pi^2n}{8a_n^2}\left(1+\sqrt{1-(1-2/n)^2}\right)\left(1-\frac{a_n^2}{n}\right)+ n\log 2+\frac{5}{4}\log n\right\}\\
& = \exp\left\{-\frac{\pi^2n}{8a_n^2}\left(1+\frac{2}{\sqrt{n}}-\frac{a_n^2}{n}+o\left(\frac{1}{\sqrt{n}}\right)\right)+ n\log 2+\frac{5}{4}\log n\right\}\\
&\leq\exp\left\{-\frac{\pi^2n}{8a_n^2}\left(1+\frac{1}{\sqrt{n}}\right)+ n\log 2+\frac{5}{4}\log n\right\},  
\end{align*}
where the last step uses $a_n< n^{1/4}/\sqrt{4\log n}$ and holds for all large $n$.  
Meanwhile, recall that $Z_n(a)$ has its expectation equal to 
\begin{align*}
\E[Z_n(a_n)] = \left(\frac{4}{\pi}+o(1)\right)\exp\left\{-\frac{\pi^2n}{8a_n^2}+ n\log 2\right\}. 
\end{align*}
Taking the ratio between $\E[\bar{Y}_n(a_n)]$ and $\E[Z_n(a_n)]$ and noting $a_0<a_n<n^{1/4}/\sqrt{4\log n}$,     
\begin{align}
\frac{\E[\bar{Y}_n(a_n)]}{\E[Z_n(a_n)]}\leq\exp\left\{-\frac{\pi^2\sqrt{n}}{8a_n^2}+\frac{5}{4}\log n\right\} \leq n^{-3}.\label{ratio-expec} 
\end{align}
This yields \eqref{expec} as desired. 

When $a_n\to\infty$, we can lift the above result in expectation to a high-probability statement using standard concentration tools. 
For $Z_n(a_n)$, the result follows from Chebyshev's inequality combined with the second-moment estimates in \Cref{sec:1/2}. For any $\delta>0$, 
\begin{align}
\Pp\left\{|Z_n(a_n)-\E[Z_n(a_n)]|\geq \delta\E[Z_n(a_n)]\right\}\leq\frac{\E[Z^2_n(a_n)]-\E[Z_n(a_n)]^2}{\delta^2\E[Z_n(a_n)]^2}. 
\end{align} 
Recall \eqref{hr} in \Cref{sec:1/2} we have established that 
\begin{align*}
\frac{\E[Z^2_n(a_n)]}{\E[Z_n(a_n)]^2}=\frac{1+o(1)}{\sqrt{1-\beta/a_n^2}},
\end{align*}
where $o(1)$ depends on $n$. Therefore, there exists a positive sequence $c_n\to 0$ such that  
\begin{align}
\Pp\left\{|Z_n(a_n)-\E[Z_n(a_n)]|\geq \delta\E[Z_n(a_n)]\right\}\leq\frac{1}{\delta^2}\left(\frac{1+c_n}{\sqrt{1-\beta/a_n^2}}-1\right), \label{event1}
\end{align}
where $o(1)$ here depends only on $n$. Choosing $\delta = ((1+c_n)/\sqrt{1-\beta/a_n^2}-1)^{1/3}=o(1)$ yields the concentration result for $Z_n(a_n)$. 

To transfer the result to $Y_n(a_n)$, applying Markov's inequality and \eqref{ratio-expec}, 
\begin{align}
\Pp\left\{\bar{Y}_n(a_n)\geq \frac{1}{n}\E[Z_n(a_n)]\right\}\leq\frac{n\E[\bar{Y}_n(a_n)]}{\E[Z_n(a_n)]}\leq n^{-2}. \label{event2}
\end{align} 
Conditional on the complements of the events in \eqref{event1} and \eqref{event2}, which hold with probability tending to one, 
\begin{align*}
\bar{Y}_n(a_n)<\frac{\E[Z_n(a_n)]}{n}<\frac{2Z_n(a_n)}{n}. 
\end{align*}
Moreover, \eqref{ratio-expec} gives $\E[Y_n(a_n)]=(1-o(1))\E[Z_n(a_n)]$. On the same high-probability event,
\[
Z_n(a_n)=(1+o(1))\E[Z_n(a_n)],
\qquad
\bar Y_n(a_n)=o(\E[Z_n(a_n)]),
\]
and hence
\[
Y_n(a_n)=Z_n(a_n)-\bar Y_n(a_n)
=(1+o(1))\E[Y_n(a_n)].
\]
This proves the concentration claims. 

\subsection{Proof of \Cref{thm:ogp}}

The proof is based on a first-moment calculation combined with Markov's inequality. We first prove 2-OGP on a finite grid and then pass to the full interval by continuity. Fix a small $\e>0$ and take $\mathcal I\subset[0,\pi/2]$ to be an equally spaced grid containing both endpoints and having mesh at most $2e^{-\e n/2}$. Then $|\mathcal I|\lesssim e^{\e n/2}$. We also assume $\eta_1 = \eta_2-\delta$ and $\eta_2\in [2\delta, 0.99]$, and take $\delta = 10^{-3}$.    

Fix $a<2$ and $(\tau_1, \tau_2)\in\mathcal I\times \mathcal I$. 
Define 
\begin{align*}
q^{(i)}(t; \tau_i, \sigma^{(i)})=\frac{1}{\sqrt{n}}\sum_{j=1}^n\sigma^{(i)}_jg^{(i)}_j(t; \tau_i),\quad\quad i = 1, 2,  
\end{align*}
where $g_j^{(i)}(t; \tau_i)$ are the same as defined in \Cref{def:m-OGP}. Consider the following random variable counting the configurations measured under the 2-OGP over the discrete set $\mathcal I$: 
\begin{align*}
&X_n(a)\coloneqq\sum_{(\tau_1, \tau_2)\in\mathcal I\times\mathcal I}\sum_{\substack{\sigma^{(1)}, \sigma^{(2)}\\ \eta_1\leq |\langle\sigma^{(1)}, \sigma^{(2)}\rangle|/n\leq\eta_2 }}\mathbf 1\{\mathcal B(\sigma^{(1)}, \sigma^{(2)}, \tau_1, \tau_2)\},  
\end{align*}
where 
\begin{align*}
\mathcal B(\sigma^{(1)}, \sigma^{(2)}, \tau_1, \tau_2)\coloneqq\left\{\left\|q^{(i)}(\cdot; \tau_i, \sigma^{(i)})\right\|_\infty\leq \frac{a}{\sqrt{n}},\  \text{for $i=1, 2$}\right\}. 
\end{align*}
The 2-OGP event over $\mathcal I$ can be bounded using Markov's inequality:
\begin{align*}
\Pp\left\{\mathcal S(\eta_1, \eta_2, 2, a, \mathcal I)\neq\emptyset\right\} = \Pr\{X_n(a)>0\}\leq \E[X_n(a)]. 
\end{align*} 
The 2-OGP would follow if we can show $\E[X_n(a)]$ is exponentially small in $n$. 

Note that $(q^{(1)}(t; \tau_1, \sigma^{(1)}), q^{(2)}(t; \tau_2, \sigma^{(2)}))$ is a planar Brownian motion with constant correlation 
\begin{align*}
\rho(\sigma^{(1)}, \sigma^{(2)}, \tau_1, \tau_2)\coloneqq \frac{1}{n}\cos(\tau_1)\cos(\tau_2)\langle\sigma^{(1)}, \sigma^{(2)}\rangle\implies |\rho(\sigma^{(1)}, \sigma^{(2)}, \tau_1, \tau_2)|\leq\eta_2. 
\end{align*}  
Using \Cref{lm:corB}, we have 
\begin{align*}
\E[\mathbf 1\{\mathcal B(\sigma^{(1)}, \sigma^{(2)}, \tau_1, \tau_2)\}] &= \Pp\left\{\left(q^{(1)}(t;\tau_1,\sigma^{(1)}),q^{(2)}(t;\tau_2,\sigma^{(2)})\right)\in[-a/\sqrt n,a/\sqrt n]^2,\ \forall t\in[0,1]\right\}\\
& = C_1(\rho(\sigma^{(1)}, \sigma^{(2)}, \tau_1, \tau_2))\exp\left\{-\frac{\lambda_1(\rho(\sigma^{(1)}, \sigma^{(2)}, \tau_1, \tau_2)) n}{2a^2}\right\} (1 + o(1)). 
\end{align*}
Moreover, since $\lambda_1$ is decreasing in $(0, 1)$ and $C_1$ is continuous on $[0, 0.99]$, for all sufficiently large $n$, 
\begin{align*}
\E[\mathbf 1\{\mathcal B(\sigma^{(1)}, \sigma^{(2)}, \tau_1, \tau_2)\}]\leq 2C_1^*\exp\left\{-\frac{\lambda_1(\eta_2) n}{2a^2}\right\}, \quad\quad C_1^*\coloneqq\max_{|\rho|\leq 0.99}C_1(\rho). 
\end{align*}
Consequently, we can bound  $\E[X_n(a)]$ by linearity of expectation. Let
\[
m_n\coloneqq\left\lfloor\frac{(1-\eta_1)n}{2}\right\rfloor.
\]
For each fixed $\sigma^{(1)}$, each sign of the overlap contributes at most $\delta n/2+2$ admissible Hamming distances, and the largest corresponding binomial coefficient is at most ${n\choose m_n}$. Hence
\begin{equation}
\begin{aligned}
\E[X_n(a)] &= \sum_{(\tau_1, \tau_2)\in\mathcal I\times\mathcal I}\sum_{\sigma^{(1)}}\sum_{\sigma^{(2)}: \eta_1\leq |\langle\sigma^{(1)}, \sigma^{(2)}\rangle|/n\leq \eta_2}\E[\mathbf 1\{\mathcal B(\sigma^{(1)}, \sigma^{(2)}, \tau_1, \tau_2)\}]\\
& \leq e^{\e n}\cdot 2^n\cdot 2\left(\frac{\delta n}{2}+2\right){n\choose m_n}\cdot 2C_1^*\exp\left\{-\frac{\lambda_1(\eta_2) n}{2a^2}\right\}\\
& = \exp\left\{-\left[\frac{\lambda_1(\eta_2)}{2a^2}-\e-\left(\log 2+h\left(\frac{1}{2}+\frac{\eta_2-\delta}{2}\right)\right) + o(1)\right]n\right\}. \label{En} 
\end{aligned}
\end{equation}
For the desired exponential smallness, we need to choose $a$ sufficiently small such that 
\begin{align*}
\frac{\lambda_1(\eta_2)}{2a^2}-\left(\log 2+h\left(\frac{1}{2}+\frac{\eta_2-\delta}{2}\right)\right)>0\iff a<\sqrt{\frac{\lambda_1(\eta_2)}{2\left(\log 2+h\left(\frac{1}{2}+\frac{\eta_2-\delta}{2}\right)\right)}},   
\end{align*}
where we ignore the effect of $\e$ because it can be chosen sufficiently small at the end without changing the result. 
To further simplify, we replace $\lambda_1(\eta_2)$ by its lower bound in \eqref{newguy}:
\begin{align}
\lambda_1(\eta_2) \geq  \underline{\lambda}_1(\eta_2)\coloneqq\frac{\pi^2}{4}\left(1+\sqrt{1-\eta_2^2}\right).   
\end{align}
To choose $a$ as large as possible, we maximize over $\eta_2\in [2\delta, 0.99]$ the lower bound function 
\begin{align*}
J(\eta_2)\coloneqq\sqrt{\frac{\underline{\lambda}_1(\eta_2)}{2\left(\log 2+h\left(\frac{1}{2}+\frac{\eta_2-\delta}{2}\right)\right)}} = \sqrt{\frac{\pi^2(1+\sqrt{1-\eta_2^2})}{8\left(\log 2+h\left(\frac{1}{2}+\frac{\eta_2-\delta}{2}\right)\right)}}. 
\end{align*}
We use the explicit choice $\eta_2=0.943$, so that $\eta_1=0.942$, and take $a=1.411$ and $\e=10^{-4}$. Outward-rounded interval arithmetic gives
\[
\underline\lambda_1(0.943)>3.2885,
\qquad
h(0.971)<0.13125,
\qquad
\log2<0.693148.
\]
Consequently,
\[
\frac{\underline\lambda_1(0.943)}{2(1.411)^2}
-10^{-4}-\bigl(\log2+h(0.971)\bigr)
>1.3\times10^{-3}.
\]
Thus the exponent in \eqref{En} is at most $-10^{-3}n$ for all sufficiently large $n$. This proves the 2-OGP with parameters $(0.942,0.943,2,1.411,\mathcal I)$. 

To pass this result to the whole interval, note that for any $\tau_i\in [0, \pi/2]$, there exists $\tau_i'\in\mathcal I$ such that $|\tau_i-\tau_i'|< 2e^{-\e n/2}$. Consequently, 
\begin{align*}
|q^{(i)}(t; \tau_i, \sigma^{(i)})-q^{(i)}(t; \tau'_i, \sigma^{(i)})|&=\left|\frac{1}{\sqrt{n}}\sum_{j=1}^n\sigma^{(i)}_j(g^{(i)}_j(t; \tau_i)-g^{(i)}_j(t; \tau'_i))\right|\\
&\leq \sqrt{n}\max_{j}|g^{(i)}_j(t; \tau_i)-g^{(i)}_j(t; \tau'_i)|\\
&\leq 2\sqrt{n}\max_j\{\|f_j^{(0)}\|_\infty, \|f_j^{(i)}\|_\infty\}|\tau_i-\tau_i'|\\
&< 4\sqrt{n}e^{-\e n/2}\max_j\{\|f_j^{(0)}\|_\infty, \|f_j^{(i)}\|_\infty\}, 
\end{align*}
where the penultimate step used the Lipschitz property of $\cos(\tau)$ and $\sin(\tau)$. A standard union-bound leveraging the exponential decay of $\|f^{(0)}_1\|_\infty$ yields that with probability at least $1-n^3e^{-\e' n}$ for some absolute constant $\e'>0$, $\max_{i, j}\{\|f_j^{(0)}\|_\infty, \|f_j^{(i)}\|_\infty\}\leq \sqrt{n}$. We denote this event by $\mathcal G$. Conditional on $\mathcal G$, for all sufficiently large $n$,
\begin{align*}
\max_{i, \sigma^{(i)}}\left\|q^{(i)}(\cdot; \tau_i, \sigma^{(i)})-q^{(i)}(\cdot; \tau'_i, \sigma^{(i)})\right\|_\infty<4ne^{-\e n/2}<\frac{10^{-3}}{\sqrt n}. 
\end{align*} 
Consequently, 
\begin{align*}
\Pp\left\{\mathcal S(0.942, 0.943, 2, 1.41, [0, \pi/2])\neq\emptyset\right\}&\leq\Pp\{\mathcal G^\complement\} + \Pp\left\{\mathcal S(0.942, 0.943, 2, 1.41, [0, \pi/2])\neq\emptyset; \mathcal G\right\}\\
&\leq\Pp\{\mathcal G^\complement\} + \Pp\left\{\mathcal S(0.942, 0.943, 2, 1.411, \mathcal I)\neq\emptyset; \mathcal G\right\}\\
&\leq n^3e^{-\e' n} + e^{-cn}. 
\end{align*}
Merging terms into a single exponential term for all large $n$ finishes the proof. 

\subsection{Proof of \Cref{thm:ogp-n}}

The proof follows from a direct counting argument. Since the interpolation interval consists of only the single point $0$, there is only one instance of Brownian paths. Without loss of generality, assume $a_n<n^{1/5}$; the more general $a_n$ can be reduced to this by taking the minimum with $n^{1/5}$ and using monotonicity. For convenience, write this instance as $f_1,\ldots, f_n$. Let
\begin{align*}
\mathcal D_n\coloneqq\left\{Z_n(a_n)\geq \exp\left\{\left(\log 2-\frac{\pi^2}{8a_n^2}\right)n\right\}\right\}.
\end{align*}
The concentration conclusion of \Cref{thm:count}, together with $4/\pi>1$, gives $\Pp(\mathcal D_n)=1-o(1)$.

We claim that $\mathcal D_n$ implies $\mathcal S(\eta_1,\eta_2,2,a_n,\{0\})\neq\emptyset$ for all sufficiently large $n$. For any $\sigma,\sigma'\in\{\pm1\}^n$, denote their Hamming distance by $\dist(\sigma,\sigma')$. Since
\begin{align*}
\frac{|\langle\sigma,\sigma'\rangle|}{n}
=\frac{|n-2\dist(\sigma,\sigma')|}{n},
\end{align*}
absolute overlap in $[\eta_1,\eta_2]$ corresponds, in particular, to Hamming distances in
\[
\mathcal I_n\coloneqq\left(\frac{(1-\eta_2)n}{2},\frac{(1-\eta_1)n}{2}\right).
\]
For all large $n$, choose an even integer $d_n\in\mathcal I_n$ such that $d_n/n$ stays uniformly bounded away from both $0$ and $1$ for all large $n$; for example, take the even integer closest to the midpoint of $\mathcal I_n$.

Suppose that $\mathcal S(\eta_1,\eta_2,2,a_n,\{0\})=\emptyset$ and $\mathcal D_n$ occur simultaneously. Any two signings in
$\{\sigma:L(\sigma)\leq a_n\}$ must have Hamming distance different from $d_n$. The chosen distance is even and satisfies $\alpha n<d_n<(1-\alpha)n$ for some fixed $\alpha=\alpha(\eta_1,\eta_2)>0$. The binary-code form of the Frankl--R\"odl forbidden-distance theorem \cite[Theorem~1.10]{frankl1987forbidden} therefore gives a constant $\varepsilon_0=\varepsilon_0(\eta_1,\eta_2)>0$ such that
\[
Z_n(a_n)\leq(2-\varepsilon_0)^n.
\]
On the other hand, $\mathcal D_n$ gives
\[
Z_n(a_n)\geq
\exp\left\{\left(\log2-\frac{\pi^2}{8a_n^2}\right)n\right\}.
\]
Since $a_n\to\infty$, the latter exponential rate tends to $\log2$ and is eventually strictly larger than $\log(2-\varepsilon_0)$, a contradiction. Hence $\mathcal D_n$ implies the desired nonempty solution set, and therefore
\[
\Pp\{\mathcal S(\eta_1,\eta_2,2,a_n,\{0\})\neq\emptyset\}
\geq\Pp(\mathcal D_n)=1-o(1).
\]